\documentclass[a4paper,12pt,leqno]{article}

\usepackage[utf8]{inputenc}
\usepackage[T1]{fontenc}
\usepackage{lmodern}
\usepackage{amsmath}
\usepackage{amsthm}
\usepackage{amssymb}
\usepackage{mathrsfs}
\usepackage{hyperref}

\newcommand{\N}{\mathbb N}
\newcommand{\Z}{\mathbb Z}
\newcommand{\Q}{\mathbb Q}
\newcommand{\R}{\mathbb R}
\newcommand{\C}{\mathbb C}

\newcommand{\B}{\mathbb B}

\newcommand{\F}{\mathbb F}
\newcommand{\HH}{\mathbb H}
\newcommand{\OO}{\mathbb O}

\newcommand{\cF}{\mathcal F}

\newcommand{\sU}{\mathscr U}

\newcommand{\Lg}{\mathfrak g}
\newcommand{\LK}{K}
\newcommand{\Lk}{\mathfrak k}
\newcommand{\LN}{N}
\newcommand{\Ln}{\mathfrak n}
\newcommand{\LA}{A}
\newcommand{\La}{\mathfrak a}
\newcommand{\Hil}{\mathscr H}

\newcommand{\Ban}{\mathscr X}

\newcommand{\support}{\operatorname{supp}}

\newcommand{\co}{\operatorname{co}}
\newcommand{\sspan}{\ensuremath{\operatorname{span}}}
\newcommand{\linbeg}{\boldsymbol{\operatorname{B}}}

\newcommand{\transposed}{\mathrm T}

\newcommand{\cont}{C}
\newcommand{\vanish}{0}
\newcommand{\cstar}{\ensuremath{C^*}}
\newcommand{\cb}{\mathrm{cb}}

\newcommand{\cpt}{\mathrm c}
\newcommand{\unit}{\boldsymbol 1}
\newcommand{\dd}{\mathrm d}

\newcommand{\SO}{{SO}}
\newcommand{\PGL}{{PGL}}

\newcommand{\ELL}{{L}}
\newcommand{\trip}{|\!|\!|}
\newcommand{\ip}[2]{\langle#1,#2\rangle}
\newcommand{\dual}[2]{\langle#1,#2\rangle}
\newcommand{\seta}[1]{\ensuremath{\{#1\}}}
\newcommand{\setw}[2]{\setaw{#1}{#2}}
\newcommand{\setaw}[2]{\seta{#1\::\:#2}}
\newcommand{\Email}{\begingroup \def\UrlLeft{<}\def\UrlRight{>} \urlstyle{tt}\Url}
\newcommand{\mailto}[1]{\href{mailto:#1}{\Email{#1}}}

\newcommand{\contrib}[3]{#1\quad\mailto{#2}{\small\\\quad\textit{#3}}\\[1ex]}

\newcommand{\ST}{S}
\newcommand{\SK}{S}
\newcommand{\ccc}{\boldsymbol c}
\newcommand{\IA}{A}
\newcommand{\IN}{N}
\newcommand{\pA}{\LA^+}
\newcommand{\cpA}{\bar\LA^+}
\newcommand{\GG}{\SO_0(1,n)}
\newcommand{\KK}{\SO(n)}

\newcommand{\y}{g}
\newcommand{\FA}{A}
\newcommand{\A}{A}
\newcommand{\FSA}{B}
\newcommand{\BB}{B}

\newcommand{\FF}{(\ast_{m=1}^M\Z/2\Z)\ast(\ast_{n=1}^N\Z)}
\newcommand{\EE}{E}

\newcommand{\G}{G}
\newcommand{\K}{K}

\newcommand{\Y}{Y}

\newcommand{\PGLQ}{\PGL_2(\padicnum)}

\newcommand{\PSLZ}{\PGL_2(\padicint)}

\newcommand{\padicnum}{\Q_q}

\newcommand{\padicint}{\Z_q}

\newcommand{\bS}{\mathbb{S}_\h}
\newcommand{\bSbar}{\mathbb{\bar S}_\h}
\newcommand{\abS}{\mathbb{S}_{2a}}
\newcommand{\MA}{MA}
\newcommand{\MoA}{M_0A}
\newcommand{\h}{m}

\newcommand{\Gammaload}{\frac{\Gamma \left( \frac{\h}{2}+\sigma \right) \Gamma \left( \frac{\h}{2}-\sigma \right) \Gamma \left( \frac{\h}{2}+it \right) \Gamma \left( \frac{\h}{2}-it \right)}{\Gamma \left( \frac{\h}{2} \right) \Gamma \left( \frac{\h}{2} \right) \left| \Gamma \left( \frac{\h}{2}+s \right) \Gamma \left( \frac{\h}{2}-s \right) \right|}}
\newcommand{\Gammaloadd}{\frac{\Gamma \left( \frac{\h}{2}+\Re(s) \right) \Gamma \left( \frac{\h}{2}-\Re(s) \right) \Gamma \left( \frac{\h}{2}+i\Im(s) \right) \Gamma \left( \frac{\h}{2}-i\Im(s) \right)}{\Gamma \left( \frac{\h}{2} \right) \Gamma \left( \frac{\h}{2} \right) \left| \Gamma \left( \frac{\h}{2}+s \right) \Gamma \left( \frac{\h}{2}-s \right) \right|}}
\newcommand{\strip}{\setw{\sigma+it\in\C}{|\sigma|<\frac{\h}{2},\, t\in\R}}
\newcommand{\abstrip}{\setw{\sigma+it\in\C}{|\sigma|<a,\, t\in\R}}

\newcommand{\PP}{P}
\newcommand{\QQ}{Q}

\swapnumbers
\numberwithin{equation}{section}
\theoremstyle{plain}
\newtheorem{theorem}{Theorem}[section]
\newtheorem{maintheorem}[theorem]{Theorem}
\newtheorem{proposition}[theorem]{Proposition}
\newtheorem{lemma}[theorem]{Lemma}
\newtheorem{corollary}[theorem]{Corollary}

\theoremstyle{definition}

\theoremstyle{remark}
\newtheorem{remark}[theorem]{Remark}

\renewcommand{\vec}{\boldsymbol}
\renewcommand{\Re}{\mathrm{Re}}
\renewcommand{\Im}{\mathrm{Im}}
\renewcommand{\contrib}[2]{#1\quad\mailto{#2}}

\begin{document}
	\hyphenation{hil-bert-space fran-ces-co lo-rentz boun-ded func-tions haa-ge-rup pro-per-ties a-me-na-bi-li-ty re-wri-ting pro-per-ty tro-els steen-strup jen-sen de-note as-su-ming Przy-by-szew-ska ma-jo-ri-za-tion o-pe-ra-tors mul-ti-pli-ers i-so-mor-phism la-place bel-tra-mi re-pre-sen-ta-tions pro-po-si-tion the-o-rem lem-ma co-rol-la-ry}

	\title{Fourier Multiplier Norms of Spherical Functions on the Generalized Lorentz Groups}
\author{Troels Steenstrup\thanks{Partially supported by the Ph.D.-school OP--ALG--TOP--GEO.}}
\date{}
\maketitle

	\begin{abstract}
		Our main result provides a closed expression for the completely bounded Fourier multiplier norm of the spherical functions on the generalized Lorentz groups $\GG$ (for $n\geq2$). As a corollary, we find that there is no uniform bound on the completely bounded Fourier multiplier norm of the spherical functions on the generalized Lorentz groups. We extend the latter result to the groups $SU(1,n)$, $Sp(1,n)$ (for $n\geq2$) and the exceptional group $F_{4(-20)}$, and as an application we obtain that each of the above mentioned groups has a completely bounded Fourier multiplier, which is not the coefficient of a uniformly bounded representation of the group on a Hilbert space.

	\end{abstract}
	\section*{Introduction}
\label{intro1}
Let $\Y$ be a non-empty set. A function $\psi:\Y\times\Y\to\C$ is called a \emph{Schur multiplier} if for every operator $A=(a_{x,y})_{x,y\in\Y}\in\linbeg(\ell^2(\Y))$ the matrix $(\psi(x,y)a_{x,y})_{x,y\in\Y}$ again represents an operator from $\linbeg(\ell^2(\Y))$ (this operator is denoted by $M_\psi A$). If $\psi$ is a Schur multiplier it follows easily from the closed graph theorem that $M_\psi\in\linbeg(\linbeg(\ell^2(\Y)))$, and one referrers to $\|M_\psi\|$ as the \emph{Schur norm} of $\psi$ and denotes it by $\|\psi\|_S$.

Let $G$ be a locally compact group. In~\cite{Her:UneGeneralisationDeLaNotionDetransformeeDeFourier-Stieltjes}, Herz introduced a class of functions on $\G$, which was later denoted the class of \emph{Herz--Schur multipliers} on $\G$. By the introduction to~\cite{BF:Herz-SchurMultipliersAndCompletelyBoundedMultipliersOfTheFourierAlgebraOfALocallyCompactGroup}, a continuous function $\varphi:\G\to\C$ is a Herz--Schur multiplier if and only if the function
\begin{equation}
	\label{new0.3}
	\hat\varphi(x,y)=\varphi(y^{-1}x)\qquad(x,y\in\G)
\end{equation}
is a Schur multiplier, and the \emph{Herz--Schur norm} of $\varphi$ is given by
\begin{equation*}
	\|\varphi\|_{HS}=\|\hat\varphi\|_S.
\end{equation*}

In~\cite{DCH:MultipliersOfTheFourierAlgebrasOfSomeSimpleLieGroupsAndTheirDiscreteSubgroups} De Canni{\`e}re and Haagerup introduced the Banach algebra $\MA(G)$ of \emph{Fourier multipliers} of $G$, consisting of functions $\varphi:\G\to\C$ such that
\begin{equation*}
	\varphi\psi\in\FA(\G)\qquad(\psi\in\FA(\G)),
\end{equation*}
where $\FA(\G)$ is the \emph{Fourier algebra} of $\G$ as introduced by Eymard in~\cite{Eym:L'alg`ebreDeFourierD'unGroupeLocalementCompact} (the \emph{Fourier--Stieltjes algebra} $\FSA(\G)$ of $\G$ is also introduced in this paper). The norm of $\varphi$ (denoted \emph{$\|\varphi\|_{\MA(\G)}$}) is given by considering $\varphi$ as an operator on $\FA(\G)$. According to~\cite[Proposition~1.2]{DCH:MultipliersOfTheFourierAlgebrasOfSomeSimpleLieGroupsAndTheirDiscreteSubgroups} a Fourier multiplier of $G$ can also be characterized as a continuous function $\varphi:\G\to\C$ such that
\begin{equation*}
	\lambda(g)\stackrel{M_\varphi}{\mapsto}\varphi(g)\lambda(g)\qquad(g\in\G)
\end{equation*}
extends to a $\sigma$-weakly continuous operator (still denoted $M_\varphi$) on the group von Neumann algebra ($\lambda:\G\to\linbeg(\ELL^2(\G))$ is the \emph{left regular representation} and the group von Neumann algebra is the closure of the span of $\lambda(\G)$ in the weak operator topology). Moreover, one has $\|\varphi\|_{\MA(\G)}=\|M_\varphi\|$. The Banach algebra $\MoA(\G)$ of \emph{completely bounded Fourier multipliers} of $G$ consists of the Fourier multipliers of $G$, $\varphi$, for which $M_\varphi$ is completely bounded. In this case they put \emph{$\|\varphi\|_{\MoA(\G)}=\|M_\varphi\|_{\cb}$}.

In~\cite{BF:Herz-SchurMultipliersAndCompletelyBoundedMultipliersOfTheFourierAlgebraOfALocallyCompactGroup} Bo{\.z}ejko and Fendler show that the completely bounded Fourier multipliers coincide isometrically with the continuous Herz--Schur multipliers. In~\cite{Jol:ACharacterizationOfCompletelyBoundedMultipliersOfFourierAlgebras} Jolissaint gives a short and self-contained proof of the result from~\cite{BF:Herz-SchurMultipliersAndCompletelyBoundedMultipliersOfTheFourierAlgebraOfALocallyCompactGroup} in the form stated below.
\begin{proposition}[\cite{BF:Herz-SchurMultipliersAndCompletelyBoundedMultipliersOfTheFourierAlgebraOfALocallyCompactGroup},~\cite{Jol:ACharacterizationOfCompletelyBoundedMultipliersOfFourierAlgebras}]
	\label{Gilbert0}
	Let $G$ be a locally compact group and assume that $\varphi:G\to\C$ and $k\geq0$ are given, then the following are equivalent:
	\begin{itemize}
		\item [(i)]$\varphi$ is a completely bounded Fourier multiplier of $\G$ with $\|\varphi\|_{\MoA(G)}\leq k$.
		\item [(ii)]$\varphi$ is a continuous Herz--Schur multiplier on $\G$ with $\|\varphi\|_{HS}\leq k$.
		\item [(iii)]There exists a Hilbert space $\Hil$ and two bounded, continuous maps $P,Q:G\to\Hil$ such that
		\begin{equation*}
			\varphi(y^{-1}x)=\ip{P(x)}{Q(y)}\qquad(x,y\in G)
		\end{equation*}
		and
		\begin{equation*}
			\|P\|_\infty\|Q\|_\infty\leq k,
		\end{equation*}
		where
		\begin{equation*}
			\|P\|_\infty=\sup_{x\in G}\|P(x)\|\quad\mbox{and}\quad\|Q\|_\infty=\sup_{y\in G}\|Q(y)\|.
		\end{equation*}
	\end{itemize}
\end{proposition}

Let $\G$ be a locally compact group and $\K$ a compact subgroup. A function $f$ on $\G$ is called \emph{$\K$-bi-invariant} if
\begin{equation*}
	f(k g k')=f(g)\qquad(g\in\G,\,k,k'\in\K).
\end{equation*}
Let $\cont_\cpt(\G)^\natural$ denote the set of compactly supported continuous functions on $\G$ which are $\K$-bi-invariant (throughout, we let the superscripts $\natural$ on a set of functions on $\G$ denote the subset consisting of the $\K$-bi-invariant functions---in general, there should be no confusion over which $\K$ is meant). The pair $(\G,\K)$ is a \emph{Gelfand pair} if $\cont_\cpt(\G)^\natural$ is commutative with respect to convolution. This implies that $\ELL^1(\G)^\natural$ is commutative with respect to convolution and that $\G$ is unimodular (cf.~\cite{CEFRT:AnalyseHarmonique}).

A \emph{spherical function} on a Gelfand pair $(\G,\K)$ is a function $\varphi\in\cont(\G)^\natural$ such that
\begin{equation*}
	f\mapsto\dual{f}{\varphi}\qquad(f\in\cont_\cpt(\G)^\natural)
\end{equation*}
is a non-zero character, where
\begin{equation*}
	\dual{f}{\varphi}=\int_\G f(g)\varphi(g)\dd\mu_\G(g)\qquad(f\in\cont_\cpt(\G)^\natural,\,\varphi\in\cont(\G)^\natural)
\end{equation*}
and $\mu_\G$ is a left and right invariant Haar measure on $\G$.

In~\cite{DCH:MultipliersOfTheFourierAlgebrasOfSomeSimpleLieGroupsAndTheirDiscreteSubgroups} it was proved that the reduced $\cstar$-algebra of any closed discrete subgroup of the \emph{generalized Lorentz groups} $\GG$ (for $n\geq2$) have the completely bounded approximation property (CBAP). The proof relied on finding good upper bounds on the $\MoA(\G)$-norm of the spherical functions on $\GG$. The main result of section~\ref{sph} (Theorem~\ref{maintheorem1}) is an exact computation of the $\MoA(\G)$-norm of the spherical functions on $\GG$:
\begin{theorem}
	\label{T0.3}
	Let $(\G,\K)$ be the Gelfand pair with $\G=\GG$ and $\K=\KK$ for $n\geq2$ and put $m=n-1$. Let $(\varphi_s)_{s\in\C}$ denote the spherical functions on $(\G,\K)$ indexed in the same way as in~\cite[Example~4.2.4]{GV:HarmonicAnalysisOfSphericalFunctionsOnRealReductiveGroups}. Then the completely bounded Fourier multiplier norm is given by
	\begin{equation*}
		\|\varphi_s\|_{\MoA(\G)}=\Gammaloadd
	\end{equation*}
	for $|\Re(s)|<\frac{\h}{2}$, where $\Gamma$ is the \emph{Gamma function}, and
	\begin{equation*}
		\|\varphi_s\|_{\MoA(G)}=1
	\end{equation*}
	for $s=\pm\frac{m}{2}$.
\end{theorem}
The spherical functions considered in Theorem~\ref{T0.3} constitute all spherical functions on $\GG$ which are completely bounded Fourier multipliers---this is contained in Theorem~\ref{T0.4}~{\it (i)}.

The main result of~\cite{DCH:MultipliersOfTheFourierAlgebrasOfSomeSimpleLieGroupsAndTheirDiscreteSubgroups} was generalized in~\cite{CH:CompletelyBoundedMultipliersOfTheFourierAlgebraOfASimpleLieGroupOfRealRankOne} to all connected, real rank one, simple Lie groups with finite center. These Lie groups are locally isomorphic to $\GG$, $SU(1,n)$, $Sp(1,n)$ (for $n\geq2$) or to the exceptional group $F_{4(-20)}$ (cf.~\cite{Hel:DifferentialGeometry;LieGroups;AndSymmetricSpaces}). The exact value of the ${\MoA(\G)}$-norm of the spherical functions on $SU(1,n)$, $Sp(1,n)$ or $F_{4(-20)}$ are not known. In section~\ref{mt1} we prove (cf.~Theorem~\ref{boundary} and~\ref{uniformlyunbd}):
\begin{theorem}
	\label{T0.4}
	Let $\G$ be $\GG$, $SU(1,n)$, $Sp(1,n)$ (for $n\geq2$) or $F_{4(-20)}$ and let $\K$ be the corresponding maximal compact subgroup coming from the Iwasawa decomposition as in~\cite{GV:HarmonicAnalysisOfSphericalFunctionsOnRealReductiveGroups}. Let $(\varphi_s)_{s\in\C}$ be the spherical functions on $(\G,\K)$ indexed as in~\cite[Example~4.2.4]{GV:HarmonicAnalysisOfSphericalFunctionsOnRealReductiveGroups}, and put
	\begin{equation*}
		m=p+2q,
	\end{equation*}
	where $p,q$ are computed according to Table~\ref{pandq}. Then
	\begin{itemize}
		\item [(i)]$\varphi_s\in\MoA(\G)$ if and only if $|\Re(s)|<\frac{m}{2}$ or $s=\pm\frac{m}{2}$.
		\item [(ii)]$\|\varphi_s\|_{\MoA(\G)}$ is not uniformly bounded on the strip $|\Re(s)|<\frac{m}{2}$.
	\end{itemize}
\end{theorem}
\begin{table}
	\begin{minipage}[t b c]{\textwidth}
		\caption[Computation of $p$ and $q$.]{Computation of $p$ and $q$.}
		\label{pandq}
		\begin{center}
			\begin{tabular}{|c|c|c|c|}
				\hline
				$\F$ & $\G$ & $p=(n-1)\dim_\R(\F)$ & $q=\dim_\R(\F)-1$ \\
				\hline
				$\R$ & $\GG$ & $n-1$ & $0$ \\
				\hline
				$\C$ & $SU(1,n)$ & $2n-2$ & $1$ \\
				\hline
				$\HH$ & $Sp(1,n)$ & $4n-4$ & $3$ \\
				\hline
				$\OO$ & $F_{4(-20)}$ & $8$ & $7$ \\
				\hline
			\end{tabular}
		\end{center}
	\end{minipage}
\end{table}
The ``if'' part of Theorem~\ref{T0.4}~{\it (i)} was proved in~\cite{DCH:MultipliersOfTheFourierAlgebrasOfSomeSimpleLieGroupsAndTheirDiscreteSubgroups} for $\GG$ and in~\cite{CH:CompletelyBoundedMultipliersOfTheFourierAlgebraOfASimpleLieGroupOfRealRankOne} for $SU(1,n)$, $Sp(1,n)$ and $F_{4(-20)}$. Not that, according to~\cite[Proposition~1.6~(b)]{CH:CompletelyBoundedMultipliersOfTheFourierAlgebraOfASimpleLieGroupOfRealRankOne}, a spherical function, on one of the Gelfand pairs considered in Theorem~\ref{T0.4}, is a Fourier multiplier if and only if it is a completely bounded Fourier multiplier (and the two norms coincide). Hence, we could choose to formulate Theorem~\ref{T0.4} (and other theorems) in terms of Fourier multipliers instead of completely bounded Fourier multipliers. We will not do that, since completely bounded Fourier multipliers seem to be the more canonical concept (and the one we consider in section~\ref{coeff}).

Results corresponding to Theorem~\ref{T0.3} and Theorem~\ref{T0.4} are obtained in~\cite[Theorem~5.8]{HSS:SchurMultipliersAndSphericalFunctionsOnHomogeneousTrees} for the Gelfand pair $(\PGLQ,\PSLZ)$, where $\padicnum$ is the field of p-adic numbers for a prime number $q$ and $\padicint$ is the subring of p-adic integers.

Let $\G$ be one of the groups $\GG$, $SU(1,n)$, $Sp(1,n)$ (for $n\geq2$) or $F_{4(-20)}$, then $\G$ has an \emph{Iwasawa decomposition} $\G=\LK\LA\LN$ (or $\Lg=\Lk+\La+\Ln$ at the level of Lie algebras), where $\LK$ is a maximal compact subgroup, $\LA$ is an abelian subgroup and $\LN$ is a nilpotent subgroup. Since $\G$ has real rank one, $\LA$ is one dimensional and is customarily written
\begin{equation*}
	\LA=\setw{a_r}{r\in\R},
\end{equation*}
where
\begin{equation}
	\label{exprH}
	a_r=\exp(rH)
\end{equation}
for a certain $H\in\La^+$.
\begin{remark}
	There is a unique positive simple root in $\La^*$, which will be denoted $\alpha$. The reader familiar with the Iwasawa decomposition will observe that $p$ and $q$ from Table~\ref{pandq} are given by $p=\dim(\Lg_\alpha)$ and $q=\dim(\Lg_{2\alpha})$, where $\Ln=\Lg_\alpha+\Lg_{2\alpha}$ is the sum of the positive root spaces. The choice of $H\in\La^+$ is made such that $\alpha(H)=1$ (cf.~\cite[Example~4.2.4]{GV:HarmonicAnalysisOfSphericalFunctionsOnRealReductiveGroups}).
\end{remark}

For $\GG$, $\LN$ is abelian while for the remaining groups $\LN$ is step-two nilpotent. It is well known that $(\G,\K)$ is a Gelfand pair for all these groups (cf.~\cite[Corollary~1.5.6]{GV:HarmonicAnalysisOfSphericalFunctionsOnRealReductiveGroups})---it is the canonical Gelfand pair on $\G$, so we will often refer to the spherical functions on the Gelfand pair $(\G,\K)$ as the spherical functions on $\G$. The polar decomposition of $\G$ (cf.~\cite[Lemma~2.2.3]{GV:HarmonicAnalysisOfSphericalFunctionsOnRealReductiveGroups}) is given by $\G=\LK\cpA\LK$, where $\pA=\setw{ a_r}{r>0 }$ and $\cpA=\setw{a_r}{r\geq0}$. Since the spherical functions on $(\G,\K)$ are $\K$-bi-invariant they can be thought of as functions on $\cpA$ (or $\LA$, using that $a_r^{-1}=a_{-r}$ and that the spherical functions are invariant under taking inverse).

Let $(\G,\K)$ be one of the above Gelfand pairs, and put
\begin{equation}
	\label{m}
	m=p+2q
\end{equation}
and
\begin{equation}
	\label{mo}
	m_0=p+2,
\end{equation}
where $q$ and $p$ are given in Table~\ref{pandq}. According to~\cite[(4.2.23)]{GV:HarmonicAnalysisOfSphericalFunctionsOnRealReductiveGroups} the spherical function $\varphi_s$ ($s\in\C$) on $(\G,\K)$ is given by
\begin{equation}
	\label{sphfcthyper}
	\varphi_s(a_r)=F\Big(\frac{m}{4}+\frac{s}{2},\frac{m}{4}-\frac{s}{2};\frac{m+m_0}{4};-\sinh(r)^2\Big)\qquad(r\in\R),
\end{equation}
where $F(a,b;c;z)$ is the \emph{Hypergeometric function} of the complex variable $z$ with parameters $a,b,c\in\C$ as defined in \cite[\S~2.1]{EMOT:HigherTranscendentalFunctions.Vol.I}.\footnote{This Hypergeometric function is sometimes called ${}_2F_1(a,b;c;z)$ instead of just $F(a,b;c;z)$, but since we do not use any other types of generalized Hypergeometric functions we choose to omit the extra subscripts. It is defined through a power series that converges absolutely for $|z|<1$ (and also for $|z|=1$ if $\Re(a+b)<\Re(c)$). If $|z|$ exceeds $1$ in our formulas we are implicitly using an analytic continuation.} Let $\bS$ be the strip in the complex plane given by $\bS=\strip$. We list here some well known results about the spherical functions on $\G$ (general references are~\cite{GV:HarmonicAnalysisOfSphericalFunctionsOnRealReductiveGroups}, \cite{CEFRT:AnalyseHarmonique} and~\cite{Hel:GroupsAndGeometricAnalysis}):
\begin{itemize}
	\item Every spherical function on $(\G,\K)$ equals $\varphi_s$ for some $s\in\C$.
	\item $\varphi_s=\varphi_{s'}$ if and only if $s=\pm s'$.
	\item $\varphi_s=\unit$ (the constant function $1$) for $s=\pm\tfrac{\h}{2}$.
	\item $\varphi_s$ is bounded if and only if $s\in\bSbar$, and in this case $\|\varphi\|_\infty=1$.
	\item For every $g\in\G$ the map $s\mapsto\varphi_s(g)$ is analytic.
	\item $\varphi_s$ (considered as a function on $\G/\K$) is an eigenfunction of the Laplace--Beltrami operator with eigenvalue $s^2-(\tfrac{\h}{2})^2$.
\end{itemize}

By a \emph{representation} $(\pi,\Hil)$ of a locally compact group $\G$ on a Hilbert space $\Hil$ we mean a homomorphism of $\G$ into the invertible elements of $\linbeg(\G)$. A representation $(\pi,\Hil)$ of $\G$ is said to be \emph{uniformly bounded} if
\begin{equation*}
	\sup_{g\in\G}\|\pi(g)\|<\infty
\end{equation*}
and one usually writes $\|\pi\|$ for $\sup_{g\in\G}\|\pi(g)\|$. If $g\mapsto\pi(g)$ is continuous with respect to the strong operator topology on $\linbeg(\G)$ then we say that $(\pi,\Hil)$ is \emph{strongly continuous}. Let $(\pi,\Hil)$ be a strongly continuous, uniformly bounded representation of $\G$ then, according to~\cite[Theorem~2.2]{DCH:MultipliersOfTheFourierAlgebrasOfSomeSimpleLieGroupsAndTheirDiscreteSubgroups}, any \emph{coefficient of $(\pi,\Hil)$} is a continuous Herz--Schur multiplier, i.e.,
\begin{equation*}
	g\stackrel{\varphi}{\mapsto}\ip{\pi(g)\xi}{\eta}\qquad(g\in\G)
\end{equation*}
is a continuous Herz--Schur multiplier with
\begin{equation*}
	\|\varphi\|_{\MoA(\G)}\leq\|\pi\|^2\|\xi\|\|\eta\|
\end{equation*}
for any $\xi,\eta\in\Hil$ (note that this result also follows as a corollary to Proposition~\ref{Gilbert0}).

U.~Haagerup has shown that on the non-abelian free groups there are Herz--Schur multipliers which can not be realized as coefficients of uniformly bounded representations. The proof by Haagerup has remained unpublished, but Pisier has later given a different proof, cf.~\cite{Pis:AreUnitarizableGroupsAmenable?}. In section~\ref{coeff} we use~\cite[Theorem~5.8]{HSS:SchurMultipliersAndSphericalFunctionsOnHomogeneousTrees} and Theorem~\ref{T0.4}~{\it (ii)} together with a modified version of Haagerup's proof to show (cf.~Theorem~\ref{HaageruponGelfand} and Remark~\ref{withoutsoctsnow}):
\begin{theorem}
	\label{T0.5}
	Let $\G$ be a group of the form $\GG$, $SU(1,n)$, $Sp(1,n)$ (with $n\geq2$), $F_{4(-20)}$ or $\PGLQ$ (with $q$ a prime number). There is a completely bounded Fourier multiplier of $\G$ which is not the coefficient of a uniformly bounded representation of $\G$.
\end{theorem}
By permission of Haagerup, his proof for the non-abelian free groups is included in section~\ref{coeff} (cf.~Theorem~\ref{Haagerup}).

	\section{\texorpdfstring{Spherical functions on $\GG$}{Spherical functions on SO0(1,n)}}
\label{sph}
The linear transformations of $n+1$ ($n\geq2$) dimensional Minkowski space leaving invariant the quadratic form
\begin{equation*}
	-x_0^2+x_1^2+\cdots +x_n^2
\end{equation*}
consists of the real $n+1\times n+1$ matrices satisfying
\begin{equation}
	\label{linearisom}
	g^\transposed Jg=J,
\end{equation}
where $g^\transposed$ denotes the transposed of $g$ and $J$ is the $n+1\times n+1$ matrix given by
\begin{equation*}
	J=\left(
	\begin{array}{cccc}
		-1 & 0 & \cdots & 0\\
		0 & 1 & \cdots & 0\\
		\vdots & \vdots & \ddots & \vdots \\
		0 & 0 & \cdots & 1
	\end{array}
	\right).
\end{equation*}
If $g=(g_{i j})_{i,j=0}^n$ is a real $n+1\times n+1$ matrix satisfying~\eqref{linearisom}, then it is easily verified that $\det(g)=\pm1$ and $|g_{00}|=\sqrt{1+g_{10}^2+\cdots+g_{n0}^2}\geq1$. We also mention that the inverse of $g$ is given by
\begin{equation*}
	g^{-1}=Jg^\transposed J=\left(
	\begin{array}{cccc}
		g_{00} & -g_{10} & \cdots & -g_{n0}\\
		-g_{01} & g_{11} & \cdots & g_{n1}\\
		\vdots & \vdots & \ddots & \vdots \\
		-g_{0n} & g_{1n} & \cdots & g_{n n}
	\end{array}
	\right).
\end{equation*}
The \emph{generalized Lorentz group} $\GG$ consists of exactly those real $n+1\times n+1$ matrices $g=(g_{i j})_{i,j=0}^n$ satisfying~\eqref{linearisom} for which $\det(g)=1$ and $g_{00}\geq1$ (this is the same as taking the connected component containing the identity). For more details, cf.~\cite[Ch.~I~\S~1]{Tak:SurLesRepr'esentationsUnitairesDesGroupesDeLorentzG'en'eralis'es} or~\cite[\S~2]{Lip:UniformlyBoundedRepresentationsOfTheLorentzGroups}. We choose the same Iwasawa decomposition for $\GG$ as~\cite{Tak:SurLesRepr'esentationsUnitairesDesGroupesDeLorentzG'en'eralis'es} and~\cite{Lip:UniformlyBoundedRepresentationsOfTheLorentzGroups}, i.e., we let the compact group $\K$ be given by
\begin{equation*}
	K=1\times\KK,
\end{equation*}
the abelian group $\IA$ be given by
\begin{equation*}
	\IA=\setw{ a_r }{ r\in\R },\qquad a_r=
	\left(
	\begin{array}{ccccc}
		\cosh(r) & \sinh(r) & 0 & \cdots & 0\\
		\sinh(r) & \cosh(r) & 0 & \cdots & 0\\
		0 & 0 & 1 & \cdots & 0\\
		\vdots & \vdots & \vdots & \ddots & \vdots\\
		0 & 0 & 0 & \cdots & 1
	\end{array}
	\right)
\end{equation*}
and the nilpotent group $\IN$ be given by
\begin{equation*}
	\IN=\setw{ n_{\vec x} }{ \vec x\in\R^\h },\qquad n_{\vec x}=
	\left(
	\begin{array}{ccccc}
		1+\frac{\|\vec x\|^2}{2} & -\frac{\|\vec x\|^2}{2} & x_1 & \cdots & x_\h\\
		\frac{\|\vec x\|^2}{2} & 1-\frac{\|\vec x\|^2}{2} & x_1 & \cdots & x_\h\\
		x_1 & -x_1 & 1 & \cdots & 0\\
		\vdots & \vdots & \vdots & \ddots & \vdots\\
		x_\h & -x_\h & 0 & \cdots & 1
	\end{array}
	\right).
\end{equation*}
When no confusion is likely to arise, we will write $\K$ as $\KK$. It is worth noting that the maps
\begin{equation*}
	r\mapsto a_r\qquad(r\in\R)
\end{equation*}
and
\begin{equation*}
	\vec x\mapsto n_{\vec x}\qquad(\vec x\in\R^\h)
\end{equation*}
are group isomorphisms (so $\IN$ as actually abelian). To tie this up with~\eqref{exprH}, note that $a_r=\exp(rH)$, where
\begin{equation}
	\label{Hchoice}
	H=\left(
	\begin{array}{ccccc}
		0 & 1 & 0 & \cdots & 0\\
		1 & 0 & 0 & \cdots & 0\\
		0 & 0 & 0 & \cdots & 0\\
		\vdots & \vdots & \vdots & \ddots & \vdots\\
		0 & 0 & 0 & \cdots & 0
	\end{array}
	\right).
\end{equation}

In this section we will exclusively consider the Gelfand pair $(\G,\K)$, where $\G=\GG$ and $\K=\KK$, and we remind the reader that in this case $m=n-1$ and $m_0=n+1=m+2$ according to Table~\ref{pandq},~\eqref{m} and~\eqref{mo} (we will avoid using $m_0$ in this section, and instead formulate everything in terms of $m$). The spherical functions on $(\G,\K)$ have many concrete realizations. We take as starting point one such realization found in~\cite[Ch.~I~\S~3]{Tak:SurLesRepr'esentationsUnitairesDesGroupesDeLorentzG'en'eralis'es} or~\cite[\S~3.1]{GV:HarmonicAnalysisOfSphericalFunctionsOnRealReductiveGroups} (we use the same indexation of the spherical functions as the latter). Note that $\G$ leaves the forward light cone
\begin{equation}
	\label{lightcone}
	C=\setw{\vec x\in\R^{n+1} }{ -x_0^2+x_1^2+\cdots+x_n^2=0,\,x_0>0 }
\end{equation}
invariant. Moreover, the map
\begin{equation*}
	\vec\zeta\mapsto\setw{t\left(
	\begin{array}{c}
		1\\
		\vec\zeta
	\end{array}
	\right) }{ t>0 }\qquad(\vec\zeta\in S^\h)
\end{equation*}
is a bijection of $S^m$ (the unit sphere in $\R^{m+1}=\R^n$) onto the set of rays in the light cone $C$. Therefore, the action of $\G$ on $C$ induces an action of $\G$ on $S^m$ . Concretely, if $g\in\G$ and $\vec\zeta\in S^m$, then $g\vec\zeta\in S^\h$ is given by
\begin{equation*}
	(g\vec\zeta)_p=\Big(g_{00}+\sum_{q=1}^n g_{0q}\zeta_q\Big)^{-1}\Big(g_{p0}+\sum_{q=1}^n g_{p q}\zeta_q\Big)\qquad(p=1,\ldots,n).
\end{equation*}
This action can also be introduced using the Iwasawa decomposition in a way that explains the following notation which we will adopt (for further explanation the reader is referred to~\cite[p.~323]{Tak:SurLesRepr'esentationsUnitairesDesGroupesDeLorentzG'en'eralis'es})
\begin{equation*}
	r(g\vec\zeta)=\ln\Big(g_{00}+\sum_{q=1}^n g_{0q}\zeta_q\Big)\qquad(g\in\G,\,\vec\zeta\in S^\h),
\end{equation*}
which makes sense since $g$ leaves invariant the forward light cone, from which it follows that $g_{00}+\sum_{q=1}^n g_{0q}\zeta_q>0$. The series of representations considered in~\cite[Theorem~3.1]{Tak:SurLesRepr'esentationsUnitairesDesGroupesDeLorentzG'en'eralis'es} will be the starting point for the investigations in this section. For $s\in\C$ let $(\rho_s,\ELL^2(S^\h))$ be the representation given by
\begin{equation}
	\label{rhoN}
	(\rho_s(g)f)(\vec\zeta)=e^{-\left(\frac{m}{2}+s\right)r(g^{-1}\vec\zeta)}f(g^{-1}\vec\zeta)\qquad(\vec\zeta\in S^\h,\,g\in\G)
\end{equation}
for $f\in\ELL^2(S^\h)$. These are strongly continuous representations and when $s\in i\R$ they are also unitary and irreducible. We mention that it is sometimes preferable to introduce these representations on the Hilbert space $\ELL^2(\IN)$ (cf.~\cite[\S~5.3]{CCJJV:GroupsWithTheHaagerupProperty}). The change from the Hilbert space $\ELL^2(S^\h)$ to $\ELL^2(\IN)$ is implemented by the stereographic projection of $S^\h$ on $\R^{n+1}$---it will be written up explicitly later in this section (cf.~Lemma~\ref{unitaries}).

For $s\in\C$ let $\varphi_s$ be given by the coefficient
\begin{equation}
	\label{startingpoint}
	\varphi_s(g)=\ip{\rho_s(g)\unit}{\unit}_{\ELL^2(S^\h)}\qquad(g\in\G),
\end{equation}
where $\unit$ denotes the constant function $1$ on $S^\h$. It is well known (cf.~\cite{Tak:SurLesRepr'esentationsUnitairesDesGroupesDeLorentzG'en'eralis'es} or~\cite{GV:HarmonicAnalysisOfSphericalFunctionsOnRealReductiveGroups}) that this definition agrees with~\eqref{sphfcthyper}, but---for the convenience of the reader---we include a proof of the following proposition which starts from~\eqref{startingpoint} and ends with~\eqref{sphfcthyper}.
\begin{proposition}
	\label{usefulformula}
	For $s\in\C$ we have
	\begin{eqnarray*}
		\varphi_s(a_r) & = & \frac{\Gamma\big(\frac{m+1}{2}\big)}{\sqrt \pi \Gamma\big(\frac{m}{2}\big)} \int_0^\pi \frac{\sin(\theta)^{m-1}}{\big(\cosh(r)+\sinh(r)\cos(\theta)\big)^{s+\frac{m}{2}}}\dd\theta\\
		& = & e^{-\left(\frac{m}{2}+s\right)r}F\left(\frac{m}{2}+s,\frac{m}{2};m;1-e^{-2r}\right)\\
		& = & F\Big(\frac{m}{4}+\frac{s}{2},\frac{m}{4}-\frac{s}{2};\frac{m+1}{2};-\sinh(r)^2\Big)
	\end{eqnarray*}
	for $r\in\R$, where $\varphi_s$ is given by~\eqref{startingpoint}.
\end{proposition}
\begin{proof}
	From~\eqref{startingpoint} it is elementary to verify the first expression for $\varphi_s(a_r)$, but we simply give a reference to~\cite[Ch.~I~\S~3~(17)]{Tak:SurLesRepr'esentationsUnitairesDesGroupesDeLorentzG'en'eralis'es}. Using the substitution $\cos(\theta)=1-2t$ we find that
	\begin{eqnarray*}
		\frac{\sqrt \pi \Gamma\left(\frac{m}{2}\right)}{\Gamma\left(\frac{m+1}{2}\right)}\varphi_{s}(a_r) 
		& = & \int_0^1 \frac{(4t(1-t))^\frac{m-1}{2}(t(1-t))^{-\frac{1}{2}}}{\left(\cosh(r)+\sinh(r)(1-2t)\right)^{s+\frac{m}{2}}}\dd t\\
		& = & 2^{m-1} e^{-\left(s+\frac{m}{2}\right)r}\int_0^1 \frac{t^{\frac{m}{2}-1}(1-t)^{\frac{m}{2}-1}}{\left(1-(1-e^{-2r})t)\right)^{s+\frac{m}{2}}}\dd t\\
		& = & \frac{2^{m-1}\Gamma\left(\frac{m}{2}\right)^2}{\Gamma(m)}e^{-\left(s+\frac{m}{2}\right)r} F\left(\frac{m}{2}+s,\frac{m}{2};m;1-e^{-2r}\right),
	\end{eqnarray*}
	where the last equality follows from
	\begin{equation*}
		F\left(a,b;c;z\right)=\frac{\Gamma\left(c\right)}{\Gamma\left(b\right)\Gamma\left(c-b\right)}\int_0^1t^{b-1}(1-t)^{c-b-1}(1-t z)^{-a}\dd t,
	\end{equation*}
	which holds for $z\in\C\setminus[1,\infty[$
	and $\Re(c)>\Re(b)>0$ (cf.~\cite[\S~2.1~(10)]{EMOT:HigherTranscendentalFunctions.Vol.I}). Using Legendre's duplication formula,
	\begin{equation}
		\label{duplication}
		\Gamma(2z)=\frac{2^{2z-1}}{\sqrt{\pi}}\Gamma (z)\Gamma\Big(z+\frac{1}{2}\Big)
	\end{equation}
	(cf.~\cite[\S~1.2~(15)]{EMOT:HigherTranscendentalFunctions.Vol.I}), with $2z=m$ we arrive at the second expression for $\varphi_s(a_r)$.
	
	We continue from the second expression for $\varphi_s(a_r)$ in order to obtain the last one. Since
	\begin{equation*}
		\frac{4z}{(1+z)^2}=1-e^{-2r}\iff z=\tanh\left(\frac{r}{2}\right)
	\end{equation*}
	we find, using
	\begin{equation*}
		F\Big(a,b;2b;\frac{4z}{(1+z)^2}\Big)=(1+z)^{2a}F\Big(a,a+\frac{1}{2}-b;b+\frac{1}{2};z^2\Big)
	\end{equation*}
	(cf.~\cite[\S~2.1~(24)]{EMOT:HigherTranscendentalFunctions.Vol.I}), that
	\begin{eqnarray*}
		\varphi_{s}(a_r) & = & \frac{\left(1+\tanh\left(\frac{r}{2}\right)\right)^{m+2s}}{e^{\left(\frac{m}{2}+s\right)r}}F\Big(\frac{m}{2}+s,\frac{1}{2}+s;\frac{m+1}{2};\tanh\left(\frac{r}{2}\right)^2\Big)\\
		& = & \cosh\left(\frac{r}{2}\right)^{-m-2s}F\Big(\frac{m}{2}+s,\frac{1}{2}+s;\frac{m+1}{2};\tanh\left(\frac{r}{2}\right)^2\Big).
	\end{eqnarray*}
	Since
	\begin{equation*}
		z=\tanh\left(\frac{r}{2}\right)^2 \iff \frac{z}{z-1}=-\sinh\left(\frac{r}{2}\right)^2
	\end{equation*}
	and
	\begin{equation*}
		1-\tanh\left(\frac{r}{2}\right)^2=\frac{1}{\cosh\left(\frac{r}{2}\right)^2}
	\end{equation*}
	we find, using
	\begin{equation*}
		F\left(a,b;c;z\right)=(1-z)^{-a}F\Big(a,c-b;c;\frac{z}{z-1}\Big)
	\end{equation*}
	(cf.~\cite[\S~2.1~(22)]{EMOT:HigherTranscendentalFunctions.Vol.I}), that
	\begin{equation*}
		\varphi_{s}(a_r) = F\Big(\frac{m}{2}+s,\frac{m}{2}-s;\frac{m+1}{2};-\sinh\left(\frac{r}{2}\right)^2\Big).
	\end{equation*}
	Since
	\begin{equation*}
		z=-\sinh\left(\frac{r}{2}\right)^2 \iff 4z(1-z)=-\sinh(r)^2
	\end{equation*}
	we find, using
	\begin{equation*}
		F\Big(a,b;a+b+\frac{1}{2};4z(1-z)\Big)=F\Big(2a,2b;a+b+\frac{1}{2};z\Big)
	\end{equation*}
	(cf.~\cite[\S~2.1~(27)]{EMOT:HigherTranscendentalFunctions.Vol.I}), that
	\begin{equation*}
		\varphi_{s}(a_r) = F\Big(\frac{m}{4}+\frac{s}{2},\frac{m}{4}-\frac{s}{2};\frac{m+1}{2};-\sinh(r)^2\Big),
	\end{equation*}
	which is the last of the claimed formulas.
\end{proof}

We now turn our attention to the main technical goal of this section, namely to write up the spherical functions using only a single representation of $\IN\IA$ (since the spherical functions are $\K$-bi-invariant we can view them as $\K$-left-invariant functions on $\G/\K=\IN\IA$). Much of the following resembles~\cite[Ch.~5]{CCJJV:GroupsWithTheHaagerupProperty}, including several of the techniques, but the end result is independent, since our setting is a degenerate case of the one considered in~\cite[Ch.~5]{CCJJV:GroupsWithTheHaagerupProperty} (here $\IN$ is step-one nilpotent instead of step-two). We start by changing from the sphere to the plane through stereographic projection from the vector $\vec\zeta_0$ given by
\begin{equation*}
	\vec\zeta_0=
	\left(
	\begin{array}{c}
		1\\
		0\\
		\vdots\\
		0
	\end{array}
	\right).
\end{equation*}
We let $\vec x_{\vec\zeta}$ denote the stereographic projection of $\vec\zeta\in S^m\setminus\{\vec\zeta_0\}$ from $\vec\zeta_0$, which is given by
\begin{equation*}
	\vec x_{\vec\zeta}=\frac{1}{1-\zeta_1}
	\left(
	\begin{array}{c}
		\zeta_2\\
		\zeta_3\\
		\vdots\\
		\zeta_n
	\end{array}
	\right)
	\qquad(\vec\zeta\in S^m\setminus\{\vec\zeta_0\}).
\end{equation*}
The inverse of this stereographic projection is given by
\begin{equation*}
	\vec \zeta_{\vec x}=\frac{1}{\|\vec x\|^2+1}
	\left(
	\begin{array}{c}
		\|\vec x\|^2-1\\
		2x_1\\
		\vdots\\
		2x_m
	\end{array}
	\right)
	\qquad(\vec x\in\R^m).
\end{equation*}
In the following lemma we choose a family of unitaries from $\ELL^2(S^m)$ to $\ELL^2(\R^m)$, which effectuates the above stereographic projection.
\begin{lemma}
	\label{unitaries}
	For $t\in\R$
	\begin{equation*}
		U_{it}:\ELL^2(S^m)\to \ELL^2(\R^m)
	\end{equation*}
	given by
	\begin{equation*}
		(U_{it}h)(\vec x)=\bigg(\frac{\Gamma(m)}{\pi^{\frac{m}{2}}\Gamma\left(\frac{m}{2}\right)}\bigg)^{\frac{1}{2}}\left(\|\vec x\|^2+1\right)^{-it-\frac{m}{2}}h(\vec\zeta_{\vec x})\qquad(\vec x\in\R^m)
	\end{equation*}
	for $h\in \ELL^2(S^m)$, is unitary. Furthermore, we have
	\begin{equation*}
		(U_{it}^*f)(\vec\zeta)=\bigg(\frac{\pi^{\frac{m}{2}}\Gamma\left(\frac{m}{2}\right)}{\Gamma(m)}\bigg)^{\frac{1}{2}}\left(\|\vec x_{\vec\zeta}\|^2+1\right)^{it+\frac{m}{2}}f(\vec x_{\vec\zeta})\qquad(\vec\zeta\in S^m\setminus\{\vec\zeta_0\})
	\end{equation*}
	for $f\in \ELL^2(\R^m)$.
\end{lemma}
\begin{proof}
	For $t\in\R$ and $h\in \ELL^2(S^m)$ we have
	\begin{eqnarray*}
		\|h\|^2_2 & = & \int_{S^m}|h(\vec\zeta)|^2 \dd \vec\zeta\\
		& = & \frac{\Gamma\left(\frac{m}{2}+\frac{1}{2}\right)}{2\pi^{\frac{m}{2}+\frac{1}{2}}} \int_{\R^m} |h(\vec\zeta_{\vec x})|^2\left(\frac{2}{\|\vec x\|^2+1}\right)^m\dd \vec x,
	\end{eqnarray*}
	where $\dd\vec\zeta$ denotes the normalized Lebesgue measure on the sphere $S^m$ while $\dd\vec x$ denotes the Lebesgue measure on $\R^m$. The constant which shows up in the second line is one over the surface area of $S^m$. In the last line we used that
	\begin{equation*}
		\sqrt{\det\bigg(\bigg(\frac{\partial \vec\zeta_{\vec x}}{\partial x_i}\cdot\frac{\partial \vec\zeta_{\vec x}}{\partial x_j}\bigg)_{i,j=1}^m\bigg)}=\left(\frac{2}{\|\vec x\|^2+1}\right)^m.
	\end{equation*}
	Using Legendre's duplication formula (cf.~\eqref{duplication}) we find that
	\begin{equation*}
		\|h\|^2_2=\frac{\Gamma(m)}{\pi^{\frac{m}{2}}\Gamma\left(\frac{m}{2}\right)} \int_{\R^m} |h(\vec\zeta_{\vec x})|^2\left(\|\vec x\|^2+1\right)^{-m}\dd \vec x,
	\end{equation*}
	and finally
	\begin{equation*}
		\|h\|^2_2=\|U_{it}h\|^2_2.
	\end{equation*}
	To check surjectivity of $U_{it}$ and the claimed expression for $U_{it}^*$, it is enough to verify that the claimed expression for $U_{it}^*$ is in fact the inverse of $U_{it}$, which is easily done.
\end{proof}
\begin{proposition}
	\label{rep1}
	For $t\in\R$
	\begin{equation*}
		\varphi_{it}(g)=\ip{\pi_{it}(g) f_{it}}{ f_{it}}_{\ELL^2(\R^m)}\qquad (g\in\G),
	\end{equation*}
	where $(\pi_{it},\ELL^2(\R^m))$ is the strongly continuous, irreducible, unitary representation of $\G$ given by
	\begin{equation*}
		\pi_{it}(g)=U_{it}\rho_{it}(g)U_{it}^*\qquad(g\in\G),
	\end{equation*}
	and where $f_{it}$ is the $\K$-invariant norm $1$ vector in $\ELL^2(\R^m)$ given by
	\begin{equation*}
		f_{it}=U_{it}\unit.
	\end{equation*}
	More specifically,
	\begin{equation*}
		(\pi_{it}(g)f)(\vec x)=\big(\tfrac{1}{2}(\|\vec x\|^2+1)(1-(g^{-1}\vec\zeta_{\vec x})_1)e^{r(g^{-1}\vec\zeta_{\vec x})}\big)^{-it-\frac{m}{2}}f(\vec x_{g^{-1}\vec\zeta_{\vec x}})
	\end{equation*}
	for $f\in \ELL^2(\R^m)$, $\vec x\in\R^m$ and $g\in\G$, while
	\begin{equation*}
		f_{it}(\vec x)=\bigg(\frac{\Gamma(m)}{\pi^{\frac{m}{2}}\Gamma\left(\frac{m}{2}\right)}\bigg)^{\frac{1}{2}}\left(\|\vec x\|^2+1\right)^{-it-\frac{m}{2}}\qquad(\vec x\in\R^m).
	\end{equation*}
\end{proposition}
\begin{proof}
	Everything, except for the specific form of $\pi_{it}$ follows straight from the corresponding properties for the representation $\rho_{it}$. To prove the remaining we let $f\in \ELL^2(\R^m)$ be given, and note that
	\begin{eqnarray*}
		(\pi_{it}(g)f)(\vec x) & = & (U_{it}\rho_{it}(g)U_{it}^*f)(\vec x)\\
		& = & \bigg(\frac{\|\vec x\|^2+1}{\|\vec x_{g^{-1}\vec\zeta_{\vec x}}\|^2+1}e^{r(g^{-1}\vec\zeta_{\vec x})}\bigg)^{-it-\frac{m}{2}}f(\vec x_{g^{-1}\vec\zeta_{\vec x}}),
	\end{eqnarray*}
	which follows from the explicit expressions for $U_{it}$, $\rho_{it}$ and $U_{it}^*$ (cf.~Lemma~\ref{unitaries} and~\eqref{rhoN}). Generally, we have
	\begin{equation*}
		\|\vec x_{\vec\zeta}\|^2=\frac{1-\zeta_1^2}{(1-\zeta_1)^2}=\frac{1+\zeta_1}{1-\zeta_1}\qquad(\vec\zeta\in S^m\setminus\{\vec\zeta_0\}),
	\end{equation*}
	so
	\begin{equation*}
		\|\vec x_{g^{-1}\vec\zeta_{\vec x}}\|^2+1=\frac{2}{1-(g^{-1}\vec\zeta_{\vec x})_1}\qquad(\vec x\in \R^m),
	\end{equation*}
	which finishes the proof
\end{proof}
\begin{proposition}
	\label{rep1effect}
	For $t\in\R$ and $f\in \ELL^2(\R^m)$
	\begin{equation*}
		(\pi_{it}(a_r)f)(\vec x)=e^{-\left(it+\tfrac{m}{2}\right)r}f(e^{-r}\vec x)\qquad(\vec x\in\R^m)
	\end{equation*}
	for $r\in\R$, and
	\begin{equation*}
		(\pi_{it}(n_{\vec y})f)(\vec x)=f(\vec x-\vec y)\qquad(\vec x\in\R^m)
	\end{equation*}
	for $\vec y\in\R^m$.
\end{proposition}
\begin{proof}
	This follows from Proposition~\ref{rep1} via easy (but tedious) calculations.
\end{proof}
From Proposition~\ref{rep1effect} it is easily seen that for $t\in\R$ the representation $\pi_{it}|_{\IN\IA}$ considered here corresponds to the representation $\pi_{-it}$ considered in~\cite[p.~72]{CCJJV:GroupsWithTheHaagerupProperty}.
\begin{proposition}
	\label{rep2}
	For $t\in\R$
	\begin{equation*}
		\varphi_{it}(g)=\ip{\hat\pi_{it}(g) \hat{f_{it}}}{\hat{f_{it}}}_{\ELL^2(\R^m)}\qquad (g\in\G),
	\end{equation*}
	where $(\hat\pi_{it},\ELL^2(\R^m))$ is the strongly continuous, irreducible, unitary representation of $\G$ given by
	\begin{equation*}
		\hat\pi_{it}(g)=\cF\pi_{it}(g)\cF^*\qquad(g\in\G),
	\end{equation*}
	and where $\hat f_{it}$ is the $\K$-invariant norm $1$ vector in $\ELL^2(\R^m)$ given by
	\begin{equation*}
		\hat f_{it}=\cF f_{it},
	\end{equation*}
	where $\cF$ is the Fourier--Plancherel transform on $\ELL^2(\R^m)$. More specifically,
	\begin{equation*}
		\hat f_{it}(\vec y)=\left(\frac{\Gamma(m)}{\pi^{\frac{m}{2}}\Gamma\left(\frac{m}{2}\right)}\right)^{\frac{1}{2}}\frac{2^{1-\frac{m}{2}}}{\Gamma\left(\frac{m}{2}+it\right)}\left(\frac{\|\vec y\|}{2}\right)^{it}K_{it}(\|\vec y\|)\qquad(\vec y\in\R^m\setminus\{\vec 0\}),
	\end{equation*}
	where $K_\nu (z)$ is the \emph{modified Bessel function of the second kind} of order $\nu\in\C$ in the variable $z\in\C\setminus\{0\}$ as defined in \cite[\S~7.2]{EMOT:HigherTranscendentalFunctions.Vol.II}.\footnote{In the reference given, $K_\nu$ is called the modified Bessel function of the third kind.}
\end{proposition}
\begin{proof}
	The only nontrivial part of the proposition is the explicit formula for $\hat f_{it}$ which we will now prove. From Proposition~\ref{rep1} we know that
	\begin{equation*}
		f_{it}(\vec x)=c_m(\|\vec x\|^2+1)^{-it-\frac{m}{2}}\qquad(\vec x\in\R^m)
	\end{equation*}
	for $t\in\R$, where we write $c_m$ instead of $\big(\frac{\Gamma(m)}{\pi^{\frac{m}{2}}\Gamma(\frac{m}{2})}\big)^{\frac{1}{2}}$ for notational convenience. Unfortunately, $f_{it}\in \ELL^2(\R^m)\setminus L^1(\R^m)$, so it is not trivial to obtain $\hat f_{it}$. To see that $f_{it}\notin L^1(\R^m)$ one can use the substitutions $r=\|\vec x\|$ and $u=r^2$ in
	\begin{eqnarray*}
		\|f_{it}\|_1 & = & c_m \int_{\R^m} (\|\vec x\|^2+1)^{-\frac{m}{2}} \dd\vec x\\
		& = & c'_m \int_0^\infty (r^2+1)^{-\frac{m}{2}}r^{m-1} \dd r\\
		& = & \frac{c'_m}{2} \int_0^\infty \left(u+1\right)^{-\frac{m}{2}}u^{\frac{m}{2}-1} \dd u=\infty,
	\end{eqnarray*}
	where $c'_m$ is the (strictly positive) constant which equals $c_m$ times the surface area of the $m-1$ dimensional sphere $S^{m-1}$ in $\R^m$.	For $s=\sigma+it\in\C$ we let
	\begin{equation*}
		f_{\sigma+it}(\vec x)=c_m(\|\vec x\|^2+1)^{-\sigma-it-\frac{m}{2}}\qquad(\vec x\in\R^m)
	\end{equation*}
	be a perturbation of $f_{it}$ by $\sigma$. Using similar calculations as above we find that
	\begin{equation*}
		\|f_{\sigma+it}\|_2^2=\frac{c_mc'_m}{2}B\Big(\frac{m}{2},\frac{m}{2}+2\sigma\Big)\quad\mbox{and}\quad \|f_{\sigma+it}\|_1=\frac{c'_m}{2}B\Big(\frac{m}{2},\sigma\Big)
	\end{equation*}
	for $\sigma>0$, where $B(a,b)$ is the \emph{Beta function} in two complex variables $a,b$ with strictly positive real part, as defined in \cite[\S~1.5]{EMOT:HigherTranscendentalFunctions.Vol.I}. Since the Beta function is finite, we conclude that $f_{\sigma+it}\in \ELL^2(\R^m)\cap L^1(\R^m)$ for $\sigma>0$. We will now show that $\lim_{\sigma\to 0^+}\|f_{\sigma+it}-f_{it}\|_2=0$, from which it follows that $\lim_{\sigma\to 0^+}\|\hat f_{\sigma+it}-\hat f_{it}\|_2=0$, which is our starting point for finding $\hat f_{it}$. We find
	\begin{eqnarray*}
		\|f_{\sigma+it}-f_{it}\|_2^2
		& = & c_m^2\int_{\R^m}\big((\|\vec x\|^2+1)^{-\frac{m}{2}}((\|\vec x\|^2+1)^{-\sigma}-1) \big)^2\dd \vec x\\
		& = & c_mc'_m\int_0^\infty(r^2+1)^{-m}\big(1-(r^2+1)^{-\sigma}\big)^2 r^{m-1}\dd r,
	\end{eqnarray*}
	where we note that the second term is bounded by $1$ and converges point-wise to $0$, when $\sigma$ converges to $0$ from the right. We can therefore use Lebesgue's dominated convergence theorem to conclude that
	\begin{equation*}
		\lim_{\sigma\to 0^+}\|f_{\sigma+it}-f_{it}\|_2=0,
	\end{equation*}
	since we as integrable dominator can use
	\begin{equation*}
		r\mapsto c_mc'_m(r^2+1)^{-m}r^{m-1}\qquad (r\in\R^+),
	\end{equation*}
	whose integral equals
	\begin{equation*}
		\frac{c_mc'_m}{2}B\left(\frac{m}{2},\frac{m}{2}\right).
	\end{equation*}

	We now turn our attention to finding $\hat f_{\sigma+it}$ for $\sigma>0$, and get
	\begin{equation*}
		\hat f_{\sigma+it}(\vec y)=\frac{c_m}{(2\pi)^{\frac{m}{2}}}\int_{\R^m}(\|\vec x\|^2+1)^{-\sigma-it-\frac{m}{2}}e^{-i\ip{\vec x}{\vec y}}\dd \vec x\qquad(\vec y\in\R^m),
	\end{equation*}
	where we immediately notice that $\hat f_{\sigma+it}$ only depends on $\|\vec y\|$, since both the Lebesgue measure and the inner product in $\R^m$ are invariant under rotation. We therefore find (by rotating $\vec y$ into the first coordinate)
	\begin{equation*}
		\hat f_{\sigma+it}(\vec y)=\frac{c_m}{(2\pi)^{\frac{m}{2}}}\int_{-\infty}^\infty h_{\sigma+it}(x_1)e^{-ix_1\|\vec y\|}\dd x_1\qquad(\vec y\in\R^m),
	\end{equation*}
	where
	\begin{equation*}
		h_{\sigma+it}(x_1)=
		\left\{
		\begin{array}{lll}
			\int_{\R^{m-1}}\left(\|\vec x\|^2+1\right)^{-\sigma-it-\frac{m}{2}}\dd x_2 \cdots \dd x_m & \mbox{if} & m>1\\
			\left(x_1^2+1\right)^{-\sigma-it-\frac{1}{2}} & \mbox{if} & m=1.
		\end{array}
		\right.
	\end{equation*}
	For $m>1$ we find, using first the substitution $x_i'=\frac{x_i}{\sqrt{x_1^2+1}}$ for $i=2,\ldots,m$, then the substitution $r=\sqrt{(x'_2)^2+\cdots (x'_m)^2}$, and finally the substitution $u=r^2$:
	\begin{eqnarray*}
		h_s(x_1)
		& = & (x_1^2+1)^{-s-\frac{1}{2}}\int_{\R^{m-1}}\big((x'_2)^2+\cdots+(x'_m)^2+1\big)^{-s-\frac{m}{2}}\dd x'_2 \cdots \dd x'_m\\
		& = & (x_1^2+1)^{-s-\frac{1}{2}}\frac{c'_{m-1}}{c_{m-1}}\int_0^\infty(r^2+1)^{-s-\frac{m}{2}}r^{m-2}\dd r\\
		& = & (x_1^2+1)^{-s-\frac{1}{2}}\frac{c'_{m-1}}{2c_{m-1}}\int_0^\infty(u+1)^{-s-\frac{m}{2}}u^{\frac{m-1}{2}-1}\dd u\\
		& = & (x_1^2+1)^{-s-\frac{1}{2}}\frac{c'_{m-1}}{2c_{m-1}}B\Big(\frac{m}{2}-\frac{1}{2},\frac{1}{2}+s\Big)\\
		& = & (x_1^2+1)^{-s-\frac{1}{2}}\frac{c'_{m-1}}{2c_{m-1}}\frac{\Gamma\left(\frac{m}{2}-\frac{1}{2}\right)\Gamma\left(\frac{1}{2}+s\right)}{\Gamma\left(\frac{m}{2}+s\right)}
	\end{eqnarray*}
	for $x_1\in\R$, where we remember that $\frac{c'_{m-1}}{c_{m-1}}$ is just the surface area of the $m-2$ dimensional sphere $S^{m-2}$ in $\R^{m-1}$. Before we can put it all together, we need the following integral equation
	\begin{equation*}
		K_\nu(z)=\frac{\Gamma\left(\nu+\frac{1}{2}\right)2^{\nu-1}}{\pi^{\frac{1}{2}}z^\nu}\int_0^\infty \frac{e^{i z t}+e^{-i z t}}{\left(t^2+1\right)^{\nu+\frac{1}{2}}}\dd t\qquad(z>0)
	\end{equation*}
	for $\Re(\nu)>-\frac{1}{2}$ (cf.~\cite[(9.6.25)]{AS:HandbookOfMathematicalFunctionsWithFormulas;Graphs;AndMathematicalTables}). It now easily follows that
	\begin{equation*}
		K_\nu(z)=\frac{\Gamma\left(\nu+\frac{1}{2}\right)2^{\nu-1}}{\pi^{\frac{1}{2}}z^\nu}\int_{-\infty}^\infty \frac{e^{-i z t}}{\left(t^2+1\right)^{\nu+\frac{1}{2}}}\dd t\qquad(z>0)
	\end{equation*}
	for $\Re(\nu)>-\frac{1}{2}$. For $m>1$ and $\Re(s)>0$ we find
	\begin{eqnarray*}
		\hat f_s(\vec y)
		& = & \frac{c_m}{(2\pi)^{\frac{m}{2}}}\frac{c'_{m-1}}{2c_{m-1}}\frac{\Gamma\left(\frac{m}{2}-\frac{1}{2}\right)\Gamma\left(\frac{1}{2}+s\right)}{\Gamma\left(\frac{m}{2}+s\right)}\frac{\pi^{\frac{1}{2}}\|\vec y\|^s}{\Gamma\left(\frac{1}{2}+s\right)2^{s-1}}K_s(\|\vec y\|)\\
		& = & c_m\frac{c'_{m-1}}{c_{m-1}}\frac{\Gamma\left(\frac{m-1}{2}\right)}{2\pi^{\frac{m-1}{2}}}\frac{2^{1-\frac{m}{2}}}{\Gamma\left(\frac{m}{2}+s\right)}\left(\frac{\|\vec y\|}{2}\right)^sK_s(\|\vec y\|)\\
		& = & \bigg(\frac{\Gamma(m)}{\pi^{\frac{m}{2}}\Gamma\left(\frac{m}{2}\right)}\bigg)^{\frac{1}{2}}\frac{2^{1-\frac{m}{2}}}{\Gamma\left(\frac{m}{2}+s\right)}\left(\frac{\|\vec y\|}{2}\right)^sK_s(\|\vec y\|)
	\end{eqnarray*}
	for $\vec y\in\R^m\setminus\{\vec 0\}$, where we in the last line used that $\frac{c'_{m-1}}{c_{m-1}}=\frac{2\pi^{\frac{m-1}{2}}}{\Gamma(\frac{m-1}{2})}$ (the surface area of $S^{m-2}$) together with $c_m=\big(\frac{\Gamma(m)}{\pi^{\frac{m}{2}}\Gamma(\frac{m}{2})}\big)^{\frac{1}{2}}$. Redoing this last calculation for $m=1$, using $h_s(x_1)=(x_1^2+1)^{-s-\frac{1}{2}}$, we get precisely the same expression as substituting $m=1$ in the previous expression for $\hat f_s$.
	
	Since
	\begin{equation*}
		\lim_{\sigma\to 0^+}\|\hat f_{\sigma+it}-\hat f_{it}\|_2=0
	\end{equation*}
	we have
	\begin{equation*}
		\lim_{n\to \infty}\|\hat f_{\frac{1}{n}+it}-\hat f_{it}\|_2=0,
	\end{equation*}
	which enables us to find a subsequence which converges point-wise, that is,
	\begin{equation*}
		\lim_{k\to \infty}|\hat f_{\frac{1}{n_k}+it}(\vec y)-\hat f_{it}(\vec y)|=0
	\end{equation*}
	for almost all $\vec y\in\R^m$. From our calculation of $\hat f_{\sigma+it}$ for $\sigma>0$ we conclude that
	\begin{equation*}
		\hat f_{it}(\vec y)=\bigg(\frac{\Gamma(m)}{\pi^{\frac{m}{2}}\Gamma\left(\frac{m}{2}\right)}\bigg)^{\frac{1}{2}}\frac{2^{1-\frac{m}{2}}}{\Gamma\left(\frac{m}{2}+it\right)}\left(\frac{\|\vec y\|}{2}\right)^{it}K_{it}(\|\vec y\|)
	\end{equation*}
	for almost all $\vec y\in\R^m$, which finishes the proof.
\end{proof}
\begin{proposition}
	\label{rep2effect}
	For $t\in\R$ and $f\in \ELL^2(\R^m)$
	\begin{equation*}
		(\hat\pi_{it}(a_r)f)(\vec y)=(e^r)^{-it+\frac{m}{2}}f(e^r\vec y)\qquad(\vec y\in\R^m)
	\end{equation*}
	for $r\in\R$, and
	\begin{equation*}
		(\hat\pi_{it}(n_{\vec x})f)(\vec y)=e^{-i\ip{\vec x}{\vec y}}f(\vec y)\qquad(\vec y\in\R^m)
	\end{equation*}
	for $\vec x\in\R^m$.
\end{proposition}
\begin{proof}
	With the inversion formula in our minds, we verify the formulas on $\hat f$ instead of $f$. Using Proposition~\ref{rep1effect} and standard results on the Fourier--Plancherel transform, we find for $\hat f\in \ELL^2(\R^m)$
	\begin{equation*}
		(\hat\pi_{it}(a_r)\hat f)(\vec y)=(\pi_{it}(a_r)f)\hat{\phantom{f}}(\vec y)=(e^{-r})^{it-\frac{m}{2}}\hat f(e^r\vec y)\qquad(\vec y\in\R^m)
	\end{equation*}
	for $r\in\R$. Similarly for $\hat f\in \ELL^2(\R^m)$
	\begin{equation*}
		(\hat\pi_{it}(n_{\vec x})\hat f)(\vec y)=(\pi_{it}(n_{\vec x})f)\hat{\phantom{f}}(\vec y)=e^{-i\ip{\vec x}{\vec y}}\hat f(\vec y)\qquad(\vec y\in\R^m)
	\end{equation*}
	for $\vec x\in\R^m$.
\end{proof}
\begin{lemma}
	\label{tildeunitaries}
	For $t\in\R$
	\begin{equation*}
		\tilde U_{it}:\ELL^2(\R^m)\to \ELL^2(\R^m)
	\end{equation*}
	given by
	\begin{equation*}
		(\tilde U_{it}f)(\vec x)=2^{it}\|\vec x\|^{-it}f(\vec x)\qquad(\vec x\in\R^m)
	\end{equation*}
	for $f\in \ELL^2(\R^m)$, is unitary. Furthermore, we have
	\begin{equation*}
		(\tilde U_{it}^*f)(\vec x)=2^{-it}\|\vec x\|^{it}f(\vec x)\qquad(\vec x\in\R^m)
	\end{equation*}
	for $f\in \ELL^2(\R^m)$.
\end{lemma}
\begin{proof}
	This is obvious.
\end{proof}
\begin{proposition}
	\label{rep3}
	For $t\in\R$
	\begin{equation*}
		\varphi_{it}(g)=\ip{\tilde\pi_{it}(g) \tilde f_{it}}{\tilde f_{it}}_{\ELL^2(\R^m)}\qquad (g\in\G),
	\end{equation*}
	where $(\tilde\pi_{it},\ELL^2(\R^m))$ is the strongly continuous, irreducible, unitary representation of $\G$ given by
	\begin{equation*}
		\tilde\pi_{it}(g)=\tilde U_{it}\hat\pi_{it}(g)\tilde U_{it}^*\qquad(g\in\G),
	\end{equation*}
	and where $\tilde f_{it}$ is the $\K$-invariant norm $1$ vector in $\ELL^2(\R^m)$ given by
	\begin{equation*}
		\tilde f_{it}=\tilde U_{it} \hat f_{it}.
	\end{equation*}
	More specifically,
	\begin{equation*}
		\tilde f_{it}(\vec x)=	\bigg(\frac{\Gamma(m)}{\pi^{\frac{m}{2}}\Gamma\left(\frac{m}{2}\right)}\bigg)^{\frac{1}{2}}\frac{2^{1-\frac{m}{2}}}{\Gamma\left(\frac{m}{2}+it\right)}K_{it}(\|\vec x\|)\qquad(\vec x\in\R^m\setminus\{\vec 0\}).
	\end{equation*}
\end{proposition}
\begin{proof}
	This follows directly from Lemma~\ref{tildeunitaries} and Proposition~\ref{rep2}.
\end{proof}
\begin{proposition}
	\label{rep3effect}
	For $t\in\R$ and $f\in \ELL^2(\R^m)$
	\begin{equation*}
		(\tilde\pi_{it}(a_r)f)(\vec x)=e^{\frac{m}{2}r}f(e^r\vec x)\qquad(\vec x\in\R^m)
	\end{equation*}
	for $r\in\R$, and
	\begin{equation*}
		(\tilde\pi_{it}(n_{\vec y})f)(\vec x)=e^{-i\ip{\vec y}{\vec x}}f(\vec x)\qquad(\vec x\in\R^m)
	\end{equation*}
	for $\vec y\in\R^m$.
\end{proposition}
\begin{proof}
	For $t\in\R$ we verify the formulas on $\tilde f=\widetilde U_{it} \hat f$ instead of $f$. Using Proposition~\ref{rep2effect}, we find for $\tilde f\in \ELL^2(\R^m)$
	\begin{equation*}
		(\tilde\pi_{it}(a_r)\tilde f)(\vec x)=(\widetilde U_{it}\hat\pi_{it}(a_r)\hat f)(\vec x)=2^{it}\|\vec x\|^{-it}(e^{r})^{-it+\frac{m}{2}}\hat f(e^r\vec x)=e^{\frac{m}{2}r}\tilde f(e^r\vec x)
	\end{equation*}
	for $\vec x\in\R^m$ and $r\in\R$. Similarly for $\tilde f\in \ELL^2(\R^m)$
	\begin{equation*}
		(\tilde\pi_{it}(n_{\vec y})\tilde f)(\vec x)=(\widetilde U_{it}\hat\pi_{it}(n_{\vec y})\hat f)(\vec x)=2^{it}\|\vec x\|^{-it}e^{-i\ip{\vec y}{\vec x}}\hat f(\vec x)=e^{-i\ip{\vec y}{\vec x}}\tilde f(\vec x)
	\end{equation*}
	for $\vec x\in\R^m$ and $\vec y\in\R^m$.
\end{proof}
We have now arrived at a formulation, where the representation does not depend on $t\in\R$, as long as we only look at elements from $\IN\IA$, which we formulate in the following corollary.
\begin{corollary}
	For $t\in\R$
	\begin{equation*}
		\tilde\pi_{it}|_{\IN\IA}=\tilde\pi_0|_{\IN\IA},
	\end{equation*}
	why we shall henceforth refer to this restriction as just $\tilde\pi$. It follows that $(\tilde\pi,\ELL^2(\R^m))$ is a strongly continuous, unitary representation of $\IN\IA$.
\end{corollary}
\begin{proposition}
	\label{analytic}
	For $s\in\bS$
	\begin{equation*}
		\varphi_s(q)=\ip{\tilde\pi(q)\tilde f_s}{\tilde f_{-\bar s}}\qquad(q\in \IN\IA),
	\end{equation*}
	where
	\begin{equation*}
		\tilde f_s(\vec x)=\bigg(\frac{\Gamma(m)}{\pi^{\frac{m}{2}}\Gamma\left(\frac{m}{2}\right)}\bigg)^{\frac{1}{2}}\frac{2^{1-\frac{m}{2}}}{\Gamma\left(\frac{m}{2}+s\right)}K_s(\|\vec x\|)\qquad(\vec x\in\R^m\setminus\{\vec 0\})
	\end{equation*}
	is an element in $\ELL^2(\R^m)$, with
	\begin{equation*}
		\|\tilde f_s\|_2^2= \frac{\Gamma\left(\frac{m}{2}+\sigma\right)\Gamma\left(\frac{m}{2}-\sigma\right)\Gamma\left(\frac{m}{2}+it\right)\Gamma\left(\frac{m}{2}-it\right)}{\Gamma\left(\frac{m}{2}\right)\Gamma\left(\frac{m}{2}\right)\Gamma\left(\frac{m}{2}+s\right)\Gamma\left(\frac{m}{2}+\bar s\right)}.
	\end{equation*}
\end{proposition}
\begin{proof}
	We start by finding $\varphi_{it}(a_r n_{\vec y})$ for arbitrary $r\in\R$ and $\vec y\in\R^m$. According to Proposition~\ref{rep3} we have
	\begin{equation*}
		\varphi_{it}(a_r n_{\vec y})=\ip{\tilde\pi_{it}(a_r n_{\vec y})\tilde f_{it}}{\tilde f_{it}},
	\end{equation*}
	where according to Proposition~\ref{rep3effect}
	\begin{equation*}
		(\tilde\pi_{it}(a_r n_{\vec y})\tilde f_{it})(\vec x)=e^{\frac{m}{2}r}(\tilde\pi_{it}(n_{\vec y})\tilde f_{it})(e^r\vec x)=e^{\frac{m}{2}r}e^{-i\ip{\vec y}{e^r\vec x}}\tilde f_{it}(e^r\vec x)
	\end{equation*}
	for $\vec x\in\R^m$. Using the specific form of $\tilde f_{it}$ from Proposition~\ref{rep3} we find
	\begin{equation}
		\label{extension}
		\varphi_s(a_r n_{\vec y})=
		\frac{\pi^{-\frac{m}{2}}2^{2-m}e^{\frac{m}{2}r}\Gamma(m)}{\Gamma\left(\frac{m}{2}\right)\Gamma\left(\frac{m}{2}+s\right)\Gamma\left(\frac{m}{2}-s\right)}
		\int_{\R^m} K_{s}(e^r\|\vec x\|)K_{s}(\|\vec x\|)e^{-i e^r\ip{\vec y}{\vec x}}\dd \vec x
	\end{equation}
	for $s=it$, since $K_{-\nu}(z)=K_\nu(z)$, $\overline{K_{\nu}(z)}=K_{\bar\nu}(z)$, and $\overline{\Gamma(z)}=\Gamma(\bar z)$. As mentioned in the introduction $s\mapsto\varphi_s(g)$ is analytic for every $g\in\G$, and as such has at most one analytic continuation to $\bS$. We will now argue that the right hand side of~\eqref{extension} is in fact analytic as a functions of $s\in\bS$ (and therefore equal to $\varphi_s(a_r n_{\vec y})$ for all $s\in\bS$). Since the Gamma function is analytic, it is enough to show that the integral is analytic. Using Morera's theorem together with Cauchy's integral theorem (and an application of Fubini's theorem) one easily reduces the problem to showing continuity of the map
	\begin{equation*}
		s\mapsto\int_{\R^m} K_{s}(e^r\|\vec x\|)K_{s}(\|\vec x\|)e^{-i e^r\ip{\vec y}{\vec x}}\dd \vec x\qquad(s\in\bS).
	\end{equation*}
	We will show continuity of this map in the strips
	\begin{equation*}
		\abS=\abstrip\subset\bS\qquad(0<a<\frac{m}{2}),
	\end{equation*}
	which in turn will show continuity in the whole strip $\bS$. Continuity will follow from the continuity lemma (cf.~\cite[Theorem~11.4]{Sch:MeasuresIntegralsAndMartingales}) once we have demonstrated the existence of a dominating function, i.e., a positive function $g_a\in\ELL^1(\R^m)$ satisfying
	\begin{equation*}
		|K_{s}(e^r\|\vec x\|)K_{s}(\|\vec x\|)|\leq g_a(\vec x)\qquad(\vec x\in\R^m,\,s\in\abS).
	\end{equation*}
	We will show that
	\begin{equation*}
		g_a(\vec x)=K_{a}(e^r\|\vec x\|)K_{a}(\|\vec x\|)\qquad(\vec x\in\R^m)
	\end{equation*}
	does exactly that (for a fixed $0<a<\frac{m}{2}$).
	
	According to~\cite[(9.6.24)]{AS:HandbookOfMathematicalFunctionsWithFormulas;Graphs;AndMathematicalTables}
	\begin{equation}
		\label{Kint}
		K_\nu(z)=\int_0^\infty e^{-z\cosh(t)}\cosh(\nu t)\dd t\qquad(z>0,\,\nu\in\C).
	\end{equation}
	From~\eqref{Kint} it follows that
	\begin{equation*}
		\nu\mapsto K_\nu(z)\qquad(\nu>0)
	\end{equation*}
	is a (positive) increasing function for $z>0$. It also follows that
	\begin{equation*}
		|K_\nu(z)|\leq K_{\Re(\nu)}(z)\qquad(z>0,\,\nu\in\C)
	\end{equation*}
	since
	\begin{equation*}
		|\cosh(\nu t)|\leq\cosh(\Re(\nu t))=\cosh(\Re(\nu)t)\qquad(t>0,\,\nu\in\C).
	\end{equation*}
	Therefore, we conclude that
	\begin{equation*}
		|K_{s}(e^r\|\vec x\|)K_{s}(\|\vec x\|)|\leq K_{\Re(s)}(e^r\|\vec x\|)K_{\Re(s)}(\|\vec x\|)\leq K_{a}(e^r\|\vec x\|)K_{a}(\|\vec x\|)
	\end{equation*}
	for $s\in\abS$, from which it follows that $g_a$ is in fact a dominating function. To verify that $g_a\in\ELL^1(\R^m)$ it is enough (using the Cauchy--Schwarz inequality) to verify that $x\mapsto K_{a}(e^r\|\vec x\|)$ and $x\mapsto K_{a}(\|\vec x\|)$ both belong to $\ELL^2(\R^m)$, which is easily done using
	\begin{equation}
		\label{bigthing}
		2^{\rho+2}\Gamma(1-\rho)\int_0^\infty K_\nu(r)K_\mu(r)r^{-\rho}\dd r=
	\end{equation}
	\begin{equation*}
		\Gamma\Big(\frac{1+\nu+\mu-\rho}{2}\Big)\Gamma\Big(\frac{1+\nu-\mu-\rho}{2}\Big) \Gamma\Big(\frac{1-\nu+\mu-\rho}{2}\Big)\Gamma\Big(\frac{1-\nu-\mu-\rho}{2}\Big)
	\end{equation*}
	for $\Re(1\pm\nu\pm\mu-\rho)>0$ (cf.~\cite[\S~7.14~(36)]{EMOT:HigherTranscendentalFunctions.Vol.II}, where there is a typo in the domain requirements, which has been corrected here).
	
	Thus, we have shown that
	\begin{equation*}
		\varphi_s(q)=\ip{\tilde\pi(q)\tilde f_s}{\tilde f_{-\bar s}}\qquad(q\in \IN\IA,\,s\in\bS),
	\end{equation*}
	with $\tilde f_s$ as claimed in the proposition, and we are left with the task of finding the norm of $\tilde f_s$. Using~\eqref{bigthing} with $\nu=s=\sigma+it$, $\mu=\bar s=\sigma-it$, and $\rho=1-m$ we get
	\begin{eqnarray*}
		\|\tilde f_s\|_2^2
		& = & \frac{\Gamma(m)}{\pi^{\frac{m}{2}}\Gamma\left(\frac{m}{2}\right)}\frac{2^{2-m}}{\left|\Gamma\left(\frac{m}{2}+s\right)\right|^2}\int_{\R^m}K_s(\|\vec x\|)K_{\bar s}(\|\vec x\|)\dd \vec x\\
		& = & \frac{2^{3-m}\Gamma(m)}{\Gamma\left(\frac{m}{2}\right)^2\left|\Gamma\left(\frac{m}{2}+s\right)\right|^2}\int_0^\infty K_s(r)K_{\bar s}(r)r^{m-1}\dd r\\
		& = & \frac{\Gamma\left(\frac{m}{2}+\sigma\right)\Gamma\left(\frac{m}{2}+it\right)\Gamma\left(\frac{m}{2}-it\right)\Gamma\left(\frac{m}{2}-\sigma\right)}{\Gamma\left(\frac{m}{2}\right)^2\left|\Gamma\left(\frac{m}{2}+s\right)\right|^2}\\
		& = & \frac{\Gamma\left(\frac{m}{2}+\sigma\right)\Gamma\left(\frac{m}{2}-\sigma\right)\Gamma\left(\frac{m}{2}+it\right)\Gamma\left(\frac{m}{2}-it\right)}{\Gamma\left(\frac{m}{2}\right)\Gamma\left(\frac{m}{2}\right)\Gamma\left(\frac{m}{2}+s\right)\Gamma\left(\frac{m}{2}+\bar s\right)}.
	\end{eqnarray*}
	This finished the proof.
\end{proof}
\begin{maintheorem}
	\label{maintheorem1}
	Let $(\G,\K)$ be the Gelfand pair with $\G=\GG$ and $\K=\KK$ for $n\geq2$ and put $m=n-1$. The spherical functions $\varphi_s$ have completely bounded Fourier multiplier norm given by
	\begin{equation*}
		\|\varphi_s\|_{\MoA(G)}=\Gammaload\qquad(s\in\bS),
	\end{equation*}
	where $s=\sigma+it$, and
	\begin{equation*}
		\|\varphi_s\|_{\MoA(G)}=1\qquad(s=\pm\frac{m}{2}).
	\end{equation*}
\end{maintheorem}
\begin{proof}
	Since $\varphi_s$ is the constant function $1$ for $s=\pm\tfrac{\h}{2}$, it is trivial that $\|\varphi_{s}\|_{\MoA(G)}=1$ in this case. We will now treat the case $s\in\bS$. From Proposition~\ref{analytic} and~\ref{rep3effect} we find
	\begin{equation*}
		\varphi_s(n_{\vec y})=
		\frac{2^{2-m}\Gamma(m)}{\pi^{\frac{m}{2}}\Gamma\left(\frac{m}{2}\right)\Gamma\left(\frac{m}{2}+s\right)\Gamma\left(\frac{m}{2}-s\right)}
		\int_{\R^m} K_{s}(\|\vec x\|)^2e^{-i\ip{\vec y}{\vec x}}\dd \vec x
	\end{equation*}
	for $\vec y\in \R^m$, or just
	\begin{equation*}
		\varphi_s(n_{\vec y})=
		\int_{\R^m} h_s(\vec x)e^{-i\ip{\vec y}{\vec x}}\dd \vec x
		\qquad(\vec y\in\R^m),
	\end{equation*}
	with
	\begin{equation*}
		h_s(\vec x)=
		\frac{2^{2-m}\Gamma(m)}{\pi^{\frac{m}{2}}\Gamma\left(\frac{m}{2}\right)\Gamma\left(\frac{m}{2}+s\right)\Gamma\left(\frac{m}{2}-s\right)}K_{s}(\|\vec x\|)^2
		\qquad(\vec x\in\R^m\setminus\{\vec 0\}),
	\end{equation*}
	where $h_s\in \ELL^1(\R^m)$---we do the actual norm calculation shortly. Remember that $\LN$ is isomorphic as a group to $\R^m$ and that the dual group $\widehat \R^m$ of $\R^m$ is again $\R^m$ via the exponential map. Because of this, and the uniqueness of the Haar measure, we can interpret the expression for $\varphi_s(n_{\vec y})$ as
	\begin{equation*}
		\varphi_s|_\LN=\hat h_s,
	\end{equation*}
	with now $h_s\in \ELL^1(\widehat \LN)$ (here we use the unnormalized Fourier transform, which does not include the $(2\pi)^{-\frac{m}{2}}$ factor). From the definition of the norm on the Fourier--Stieltjes algebra (the abelian case), we have
	\begin{equation}
		\label{HHH*}
		\|\varphi_s|_{\LN}\|_{\FA(\LN)}=\|h_s\|_1,
	\end{equation}
	where
	\begin{eqnarray}
		\label{HHH**}\\
		\|h_s\|_1
		\nonumber & = & \frac{2^{2-m}\Gamma(m)}{\pi^{\frac{m}{2}}\Gamma\left(\frac{m}{2}\right)\left|\Gamma\left(\frac{m}{2}+s\right)\Gamma\left(\frac{m}{2}-s\right)\right|}\int_{\R^m} |K_{s}(\|\vec x\|)^2|\dd \vec x\\
		\nonumber & = & \frac{2^{3-m}\Gamma(m)}{\Gamma\left(\frac{m}{2}\right)^2\left|\Gamma\left(\frac{m}{2}+s\right)\Gamma\left(\frac{m}{2}-s\right)\right|}\int_0^\infty K_{s}(r)K_{\bar s}(r)r^{m-1}\dd r\\
		\nonumber & = & \Gammaload.
	\end{eqnarray}
	Here we used~\eqref{bigthing} with $\nu=s=\sigma+it$, $\mu=\bar s=\sigma-it$, and $\rho=1-m$.
	
	According to Proposition~\ref{analytic}
	\begin{equation*}
		\varphi_s(q)=\ip{\tilde\pi(q)\tilde f_s}{\tilde f_{-\bar s}}\qquad(q\in \LN\LA)
	\end{equation*}
	for $s\in\bS$, so from the definition of the norm on the Fourier--Stieltjes algebra (the non-abelian case), we have
	\begin{equation*}
		\|\varphi_s|_{\LN\LA}\|_{\FSA(\LN\LA)}\leq\|\tilde f_s\|_2\|\tilde f_{-\bar s}\|_2.
	\end{equation*}
	Using
	\begin{equation*}
		{\Gamma\left(\frac{m}{2}+s\right)\Gamma\left(\frac{m}{2}+\bar s\right)\Gamma\left(\frac{m}{2}-\bar s\right)\Gamma\left(\frac{m}{2}-s\right)}=\left|\Gamma\left(\frac{m}{2}+s\right)\Gamma\left(\frac{m}{2}-s\right)\right|^2
	\end{equation*}
	we conclude that
	\begin{equation}
		\label{HHH***}
		\|\varphi_s|_{NA}\|_{\FSA(NA)}\leq\Gammaload\qquad(s\in\bS).
	\end{equation}
	Clearly,
	\begin{equation*}
		\|\varphi_s|_{N}\|_{\FA(N)}=\|\varphi_s|_{N}\|_{\FSA(N)}\leq\|\varphi_s|_{NA}\|_{\FSA(NA)}.
	\end{equation*}
	Hence, by~\eqref{HHH*}, \eqref{HHH**} and~\eqref{HHH***}
	\begin{equation*}
		\|\varphi_s|_{NA}\|_{\FSA(NA)}=\Gammaload\qquad(s\in\bS).
	\end{equation*}
	
	Recall that $\LN\LA$ is solvable (this is part of the properties of the Iwasawa decomposition), and that solvable groups are amenable (cf.~\cite[p.~9]{Gre:InvariantMeansOnTopologicalGroupsAndTheirApplications}). Since $\varphi_s$ is $\K$-bi-invariant it now follows from \cite[Proposition~1.6~(b)]{CH:CompletelyBoundedMultipliersOfTheFourierAlgebraOfASimpleLieGroupOfRealRankOne} that $\varphi_s\in\MoA(\G)$ if and only if $\varphi_s|_{\LN\LA}\in\FSA(\LN\LA)$ and the corresponding norms coincide. This ends the proof.
\end{proof}
\begin{corollary}
	There is no uniform bound on the ${\MoA(\GG)}$-norm of the spherical functions $\varphi_s$ on the Gelfand pair $(\GG,\KK)$ for $s\in\bS$.
\end{corollary}
\begin{proof}
	This follows from Theorem~\ref{maintheorem1} by taking $s=\sigma+it$ with $t\neq0$ and $|\sigma|<\frac{m}{2}$, and observing that
	\begin{equation*}
		\lim_{\sigma\to\pm\frac{m}{2}}\|\varphi_s\|_{\MoA(G)}=+\infty,
	\end{equation*}
	since $\Gamma\big(\frac{m}{2}\mp\sigma\big)$ converges to $+\infty$ when $\sigma$ converges to $\pm\tfrac{m}{2}$, while all other $\Gamma$-terms behave nicely.
\end{proof}

	\section{Spherical functions on real rank one Lie groups}
\label{mt1}
In this section $\G$ denotes $\GG$, $SU(1,n)$, $Sp(1,n)$ (with $n\geq2$) or $F_{4(-20)}$. 
Let $\K$ be the maximal compact subgroup coming from the Iwasawa decomposition, and recall that $(\G,\K)$ is a Gelfand pair. Also, let $m,m_0\in\N$ be given by~\eqref{m} and~\eqref{mo}, respectively.

It follows from~\cite[Proposition~3.5]{DCH:MultipliersOfTheFourierAlgebrasOfSomeSimpleLieGroupsAndTheirDiscreteSubgroups} and~\cite[Theorem~4.3]{CH:CompletelyBoundedMultipliersOfTheFourierAlgebraOfASimpleLieGroupOfRealRankOne} that the spherical functions $\varphi_s$ on the Gelfand pair $(\G,\K)$ are completely bounded Fourier multipliers of $\G$ when $s\in\bS$, and according to section~\ref{sph} there there is no uniform bound on their $\MoA(\G)$-norm when $\G$ is $\GG$. In this section we focus on what happens on the border of the strip $\bS$ in the general case. From these results we will deduce that, also in the general case, there is no uniform bound on the norm $\|\varphi_s\|_{\MoA(\G)}$ for $s\in\bS$. For this, we need the asymptotic behavior of $\varphi_s(a_r)$ for $r$ going to infinity. The asymptotic behavior has been treated in~\cite[\S~13]{HC:SphericalFunctionsOnASemisimpleLieGroupI}, but we give below a simple argument anyway.
\begin{proposition}
	For $s\in\C$
	\begin{equation}
		\label{behaviour}
		\varphi_{s}(a_r) = \cosh(r)^{s-\frac{m}{2}}F\Big(\frac{m}{4}-\frac{s}{2},\frac{m_0}{4}-\frac{s}{2};\frac{m+m_0}{4};\tanh(r)^2\Big)\qquad(r\in\R).
	\end{equation}
\end{proposition}
\begin{proof}	
	Since
	\begin{equation*}
		z=-\sinh(r)^2 \iff \frac{z}{z-1}=\tanh(r)^2
	\end{equation*}
	and
	\begin{equation*}
		1+\sinh(r)^2=\cosh(r)^2
	\end{equation*}
	we find, using
	\begin{equation*}
		F\left(a,b;c;z\right)=(1-z)^{-a}F\Big(a,c-b;c;\frac{z}{z-1}\Big)
	\end{equation*}
	(cf.~\cite[\S~2.1~(22)]{EMOT:HigherTranscendentalFunctions.Vol.I}) and~\eqref{sphfcthyper}, that
	\begin{equation*}
		\varphi_{s}(a_r) = \cosh(r)^{-(s+\frac{m}{2})}F\Big(\frac{m}{4}+\frac{s}{2},\frac{m_0}{4}+\frac{s}{2};\frac{m+m_0}{4};\tanh(r)^2\Big)\qquad(r\in\R).
	\end{equation*}
	Now use that $\varphi_{-s}=\varphi_{s}$.
\end{proof}
To determine the asymptotic behavior of $\varphi_s$ for $\Re(s)\neq0$ it suffice to consider the case $\Re(s)>0$ since $\varphi_{-s}=\varphi_{s}$. In this case, the arguments of the Hypergeometric function in~\eqref{behaviour} ensures absolute convergence as a function of the last variable, when this has absolute value less than or equal to $1$. Since $\lim_{r\to\infty}\tanh(r)^2=1$, one therefore concludes that $\varphi_s(a_r)$ behaves asymptotically like
\begin{equation*}
	e^{(s-\frac{m}{2})r}2^{-s+\frac{m}{2}}F\Big(\frac{m}{4}-\frac{s}{2},\frac{m_0}{4}-\frac{s}{2};\frac{m+m_0}{4};1\Big),
\end{equation*}
when $\Re(s)>0$ and $r$ goes to plus infinity. According to~\cite[\S~2.8~(46)]{EMOT:HigherTranscendentalFunctions.Vol.I} this can be evaluated explicitly, and we find that $\varphi_s(a_r)$ behaves asymptotically like
\begin{equation}
	\label{asymptH}
	\ccc(s)e^{(s-\frac{m}{2})r}\qquad(r\in\R)
\end{equation}
when $\Re(s)>0$ and $r$ goes to plus infinity, where
\begin{equation*}
	\ccc(s)=2^{-s+\frac{m}{2}}\frac{\Gamma\left(\frac{m+m_0}{4}\right)\Gamma(s)}{\Gamma\left(\frac{m}{4}+\frac{s}{2}\right)\Gamma\left(\frac{m_0}{4}+\frac{s}{2}\right)}.
\end{equation*}
The function $\ccc$ is usually referred to as \emph{Harish-Chandra's $\ccc$-function}. We note that~\eqref{asymptH} can be found in~\cite[(4.7.24)]{GV:HarmonicAnalysisOfSphericalFunctionsOnRealReductiveGroups}.
\begin{lemma}
	\label{edge1}
	Let $(I_n)_{n\in\N}$ be a sequence of intervals $I_n=[a_n,b_n]$ in $\R$, such that $l_n=b_n-a_n$ converges to infinity as $n$ converges to infinity. If $\mu$ is a complex-valued regular measure on $\R$, then
	\begin{equation*}
		\mu(\{x_0\})=\lim_{n\to\infty}\frac{1}{l_n}\int_{I_n}e^{i r x_0}\hat\mu(r)\dd r,
	\end{equation*}
	where
	\begin{equation*}
		\hat\mu(r)=\int_{-\infty}^\infty e^{-i r x}\dd\mu(x)\qquad(r\in\R).
	\end{equation*}
\end{lemma}
\begin{proof}
	Since every complex-valued regular measure is a (complex) linear combination of at most four positive finite regular measures, we assume that $\mu$ is a positive finite regular measure. Using Fubini's theorem we find that
	\begin{eqnarray*}
		\frac{1}{l_n}\int_{I_n}e^{i r x_0}\hat\mu(r)\dd r
		& = & \frac{1}{l_n}\int_{a_n}^{b_n}\int_{-\infty}^\infty e^{-i r(x-x_0)}\dd \mu(x)\dd r\\
		& = & \int_{-\infty}^{\infty}\frac{1}{l_n}\int_{a_n}^{b_n} e^{-i r(x-x_0)}\dd r\dd \mu(x)\\
		& = & \int_{-\infty}^{\infty} f_n(x-x_0)\dd \mu(x),
	\end{eqnarray*}
	where
	\begin{equation*}
		f_n(y)=
		\left\{
		\begin{array}{lll}
			\frac{1}{l_n}\frac{e^{-i b_n y}-e^{-i a_n y}}{-i y} & \mbox{if} & y\neq 0\\
			1 & \mbox{if} & y=0
		\end{array}
		\right.
		\qquad(y\in\R).
	\end{equation*}
	Since
	\begin{equation*}
		|e^{-i b_n y}-e^{-i a_n y}|\leq 2\qquad(y\in\R)
	\end{equation*}
	we have that
	\begin{equation*}
		\lim_{n\to\infty} f_n(y)=\unit_{\{0\}}(y)\qquad(y\in\R),
	\end{equation*}
	where $\unit_{\{0\}}$ is the characteristic function on $\{0\}$. Furthermore,
	\begin{equation*}
		|e^{-i b_n y}-e^{-i a_n y}|\leq |y|l_n\qquad(y\in\R)
	\end{equation*}
	implies that we can use the constant function $1$ as an integrable dominator in Lebesgue's dominated convergence theorem, from which we find
	\begin{equation*}
		\lim_{n\to\infty}\int_{-\infty}^\infty f_n(x-x_0)\dd\mu(x)=\int_{-\infty}^\infty\unit_{\{x_0\}}\dd\mu(x)=\mu(\{x_0\}).
	\end{equation*}
\end{proof}
\begin{lemma}
	\label{edge2}
	If $\varphi$ is a continuous symmetric function on $\R$, and there exist ${x_0}\in\R\setminus \{0\}$ and $c\in\C\setminus \{0\}$ such that
	\begin{equation*}
		\lim_{r\to\infty}\varphi(r)e^{i r{x_0}}=c,
	\end{equation*}
	then $\varphi$ can not be an element in the Fourier--Stieltjes algebra $\FSA(\R)$ of $\R$.
\end{lemma}
\begin{proof}
	If $\varphi\in\FSA(\R)$, then there exists a measure $\mu\in M(\widehat\R)=M(\R)$ such that $\varphi=\hat\mu$. Since $\varphi$ is symmetric we have $\mu=\check\mu$, where
	\begin{equation*}
		\check\mu(E)=\mu(-E)\qquad(E\in\B(\R)).
	\end{equation*}
	For $n\in\N$ put $I_n=[n,2n]$, and notice that $l_n=n$ converges to infinity as $n$ converges to infinity. Since
	\begin{equation*}
		\lim_{r\to\infty}\varphi(r)e^{i r{x_0}}=c,
	\end{equation*}
	we find
	\begin{equation*}
		\lim_{r\to\infty}\frac{1}{l_n}\int_{I_n}\varphi(r)e^{i r{x_0}}\dd r=c,
	\end{equation*}
	which by Lemma~\ref{edge1} implies that $\mu(\{x_0\})=c$. Since $\check\mu=\mu$ we must have $\mu(\{-x_0\})=c$. We will show that this is not the case, and hence arrive at a contradiction.
	
	Given $\epsilon>0$ we find $n\in\N$ such that
	\begin{equation*}
		|\varphi(r)e^{i r{x_0}}-c|<\epsilon\qquad(r\geq n),
	\end{equation*}
	or equivalently
	\begin{equation*}
		|\varphi(r)e^{-i r{x_0}}-c e^{-2i r{x_0}}|<\epsilon\qquad(r\geq n),
	\end{equation*}
	and therefore
	\begin{equation*}
		\Big|\frac{1}{l_n}\int_{I_n}\varphi(r)e^{-i r{x_0}}\dd r-\frac{1}{l_n}\int_{I_n}c e^{-2i r{x_0}}\dd r\Big|<\epsilon.
	\end{equation*}
	But the last integral can easily be evaluated as
	\begin{equation*}
		\frac{1}{l_n}\int_{I_n}c e^{-2i r{x_0}}\dd r=\frac{c}{l_n}\frac{e^{-4i n{x_0}}-e^{-2i n{x_0}}}{-2i{x_0}},
	\end{equation*}
	which converges to $0$ as $n$ tends to infinity. Using this and Lemma~\ref{edge1} we find that
	\begin{equation*}
		\mu(\{-{x_0}\})=\lim_{r\to\infty}\frac{1}{l_n}\int_{I_n}\varphi(r)e^{-i r{x_0}}\dd r=0,
	\end{equation*}
	which is the desired contradiction.
\end{proof}
\begin{theorem}
	\label{boundary}
	Let $\G$ be $\GG$, $SU(1,n)$, $Sp(1,n)$ (for $n\geq2$) or $F_{4(-20)}$, then $\varphi_s\in\MoA(\G)$ if and only if $|\Re(s)|<\frac{m}{2}$ or $s=\pm\frac{m}{2}$.
\end{theorem}
\begin{proof}
	According to~\cite[Proposition~3.5]{DCH:MultipliersOfTheFourierAlgebrasOfSomeSimpleLieGroupsAndTheirDiscreteSubgroups} the spherical function $\varphi_s$ on $\GG$ is a completely bounded Fourier multiplier of $\GG$ when $|\Re(s)|<\tfrac{\h}{2}$ (this also comes out of Theorem~\ref{maintheorem1}). If $|\Re(s)|>\tfrac{\h}{2}$, then $\varphi_s$ is unbounded and therefore not a completely bounded Fourier multiplier of $\GG$. The same analysis holds for $SU(1,n)$, $Sp(1,n)$ (for $n\geq2$) and $F_{4(-20)}$ using~\cite[Theorem~4.3]{CH:CompletelyBoundedMultipliersOfTheFourierAlgebraOfASimpleLieGroupOfRealRankOne} instead of~\cite[Proposition~3.5]{DCH:MultipliersOfTheFourierAlgebrasOfSomeSimpleLieGroupsAndTheirDiscreteSubgroups}.
	
	We are left with dealing with the case $|\Re(s)|=\frac{m}{2}$. Since $\varphi_{-s}=\varphi_s$ it is enough to consider $\varphi_{s}$ for $s=\tfrac{m}{2}+it$ where $t\neq0$ (for $t=0$, $\varphi_s=\unit$ and therefore a completely bounded Fourier multiplier of $\G$). If $\varphi_s\in\MoA(\G)$, then $\varphi_s|_\LA\in\MoA(\LA)$, but since $\LA$ is abelian $\MoA(\LA)$ equals $\FSA(\LA)$ (cf.~\cite[Corollary~1.8 and Proposition~1.12]{DCH:MultipliersOfTheFourierAlgebrasOfSomeSimpleLieGroupsAndTheirDiscreteSubgroups}). Since $\G$ has real rank one, $\LA$ is isomorphic to $\R$, so we can use the the asymptotic behavior of $\varphi_s$ together with Lemma~\ref{edge2} to conclude that $\varphi_s\notin\MoA(\G)$. Specifically, we use that $\varphi_s$ is bounded together with~\eqref{asymptH} to conclude that
	\begin{equation*}
		\lim_{r\to\infty}\varphi(a_r)e^{-it r}=2^{-i t}\frac{\Gamma\left(\frac{m+m_0}{4}\right)\Gamma\left(\tfrac{m}{2}+i t\right)}{\Gamma\left(\tfrac{m}{2}+i\tfrac{t}{2}\right)\Gamma\left(\frac{m+m_0}{4}+i\tfrac{t}{2}\right)}\neq0.
	\end{equation*}
\end{proof}
\begin{maintheorem}
	\label{uniformlyunbd}
	Let $\G$ be $\GG$, $SU(1,n)$, $Sp(1,n)$ (for $n\geq2$) or $F_{4(-20)}$, then $\|\varphi_s\|_{\MoA(\G)}$ is not uniformly bounded on the strip $\bS$.
\end{maintheorem}
\begin{proof}
	We will show that if $\|\varphi_s\|_{\MoA(\G)}\leq c$ for $s\in\bS$ for a fixed $c>0$, then also $\|\varphi_s\|_{\MoA(\G)}\leq c$ for $s\in\bSbar$, which contradicts Theorem~\ref{boundary}.
	
	Recall that $s\mapsto\varphi_s(g)$ is analytic and therefore continuous for every fixed $g\in\G$ and that $\|\varphi_s\|_\infty=1$ for $s\in\bSbar$. Let $(s_n)_{n\in\N}\subseteq\bS$ be a sequence converging to $s\in\bSbar$. It follows from Lebesgue's dominated convergence theorem that $\lim_{n\to\infty}\dual{f}{\varphi_{s_n}}=\dual{f}{\varphi_s}$ for any $f\in\ELL^1(\G)$ and therefore that $\varphi_{s_n}$ converges to $\varphi_s$ in the $\sigma(\ELL^\infty(\G),\ELL^1(\G))$ topology. But according to~\cite[Lemma~1.9]{DCH:MultipliersOfTheFourierAlgebrasOfSomeSimpleLieGroupsAndTheirDiscreteSubgroups} the unit ball of $\MoA(\G)$ is $\sigma(\ELL^\infty(\G),\ELL^1(\G))$-closed. Therefore, if we assume that $\|\varphi_s\|_{\MoA(\G)}\leq c$ for every $s\in\bS$, we get that $\|\varphi_s\|_{\MoA(\G)}\leq c$ for every $s\in\bSbar$, which gives the desired contradiction.
\end{proof}

	\section{Coefficients of uniformly bounded representations}
\label{coeff}
Let $\G$ be a locally compact, unimodular group. Denote by $\mu_\G$ a fixed left- and right-invariant Haar measure on $\G$. Recall that convolution on $\ELL^1(\G)$ is given by
\begin{equation*}
	(f*h)(g')=\int_\G f(g)h(g^{-1}g')\dd\mu_\G(g)\qquad(g'\in\G)
\end{equation*}
for $f,h\in\ELL^1(\G)$, and that we have a bilinear form
\begin{equation*}
	\dual{f}{\varphi}=\int_\G f(g)\varphi(g)\dd\mu_\G(g)
\end{equation*}
for $f\in\ELL^1(\G)$ and $\varphi\in\ELL^\infty(\G)$ or $f\in\cont_\cpt(\G)$ and $\varphi\in\cont(\G)$.

For $\alpha\geq1$ we let $\ST_\alpha$ denote the set of functions $\varphi:\G\to\C$ for which there exists a strongly continuous, uniformly bounded representation $(\pi,\Hil)$ of $\G$ and vectors $\xi,\eta\in\Hil$ such that
\begin{equation*}
	\varphi(\y)=\ip{\pi(\y)\xi}{\eta}\qquad(\y\in\G),
\end{equation*}
with $\|\pi\|\leq\alpha$ and $\|\xi\|,\|\eta\|\leq1$.
\begin{lemma}
	\label{seminorm}
	For $\alpha\geq1$ and $f\in\ELL^1(\G)$ put
	\begin{equation*}
		\trip f\trip_\alpha=\sup\setw{|\dual{f}{\varphi}| }{ \varphi\in\ST_\alpha }.
	\end{equation*}
	Then $\trip\cdot\trip_\alpha$ is a Banach algebra semi-norm on the Banach convolution algebra $\ELL^1(\G)$.
\end{lemma}
\begin{proof}
	The only non-trivial part is to show
	\begin{equation*}
		\trip f*h\trip_\alpha\leq\trip f\trip_\alpha\trip h\trip_\alpha\qquad(f,h\in \ELL^1(\G)).
	\end{equation*}
	Assume that
	\begin{equation*}
		\varphi(\y)=\ip{\pi(\y)\xi}{\eta}\qquad(\y\in \G)
	\end{equation*}
	for some strongly continuous, uniformly bounded representation $(\pi,\Hil)$ of $\G$ with $\|\pi\|\leq\alpha$ and vectors $\xi,\eta\in\Hil$ with $\|\xi\|,\|\eta\|\leq1$, i.e., assume that $\varphi\in\ST_\alpha$. It follows that
	\begin{equation*}
		\dual{f}{\varphi}=\int_\G f(\y)\varphi(\y)\dd\mu_\G(\y)=\int_\G f(\y)\ip{\pi(\y)\xi}{\eta}\dd\mu_\G(\y)=\ip{\pi(f)\xi}{\eta}
	\end{equation*}
	for $f\in \ELL^1(\G)$. From this and
	\begin{equation*}
		\|\pi(f)\|=\sup\setw{|\ip{\pi(f)\xi}{\eta}|}{\xi,\eta\in\Hil,\, \|\xi\|,\|\eta\|\leq1}\qquad(f\in \ELL^1(\G))
	\end{equation*}
	it follows that $\trip f\trip_\alpha$ is the supremum of $\|\pi(f)\|$ taken over all strongly continuous, uniformly bounded representations $(\pi,\Hil)$ of $\G$ with $\|\pi\|\leq\alpha$. The Banach algebra property now follows readily, since
	\begin{equation*}
		\pi(f*h)=\pi(f)\pi(h)\qquad(f,h\in \ELL^1(\G)).\qedhere
	\end{equation*}
\end{proof}
\begin{lemma}
	\label{ClosedConvex}
	For $\alpha\geq1$, $\ST_\alpha$ is a ${\MoA(\G)}$-norm closed convex subset of $\MoA(\G)$ with $\|\varphi\|_{\MoA(\G)}\leq\alpha^2$ for $\varphi\in\ST_\alpha$.
\end{lemma}
\begin{proof}
	That $\ST_\alpha$ is a subset of $\MoA(\G)$ with $\|\varphi\|_{\MoA(\G)}\leq\alpha^2$ for $\varphi\in \ST_\alpha$ can be found in~\cite[Theorem~2.2]{DCH:MultipliersOfTheFourierAlgebrasOfSomeSimpleLieGroupsAndTheirDiscreteSubgroups} (it also follows from Proposition~\ref{Gilbert0}). The convex part is straight forward: Assume we have strongly continuous, uniformly bounded representations $(\pi_i,\Hil_i)$ of $\G$ and vectors $\xi_i,\eta_i\in\Hil_i$, such that
	\begin{equation*}
		\varphi_i(\y)=\ip{\pi_i(\y)\xi_i}{\eta_i}_{\Hil_i}\qquad(\y\in \G),
	\end{equation*}
	with $\|\pi_i\|\leq\alpha$ and $\|\xi_i\|,\|\eta_i\|\leq1$ for $i=1,2$. For $0<t<1$ we find
	\begin{equation*}
		(1-t)\varphi_1(\y)+t\varphi_2(\y)=\ip{\pi(\y)\xi}{\eta}_\Hil\qquad(\y\in \G),
	\end{equation*}
	where $(\pi,\Hil)$ is the strongly continuous, uniformly bounded representation of $\G$ given by $\pi=\pi_1\oplus\pi_2$ and $\Hil=\Hil_1\oplus\Hil_2$, while
	\begin{equation*}
		\xi=(1-t)^{\frac{1}{2}}\xi_1\oplus t^{\frac{1}{2}}\xi_2,
	\end{equation*}
	\begin{equation*}
		\eta=(1-t)^{\frac{1}{2}}\eta_1\oplus t^{\frac{1}{2}}\eta_2.
	\end{equation*}
	It is easily verified that $\|\pi\|\leq\alpha$ and $\|\xi\|,\|\eta\|\leq1$, which finishes the convex part.

	The closure part is proved using ultraproducts of Hilbert spaces. Let $\varphi$ belong to the closure of $\ST_\alpha$, and choose a sequence $(\varphi_n)_{n\in\N}$ from $\ST_\alpha$ such that
	\begin{equation*}
		\lim_{n\to\infty}\|\varphi_n-\varphi\|_{\MoA(\G)}=0.
	\end{equation*}
	This implies that
	\begin{equation*}
		\lim_{n\to\infty}|\varphi_n(\y)-\varphi(\y)|=0\qquad(\y\in \G),
	\end{equation*}
	since $\|\cdot\|_\infty\leq\|\cdot\|_{\MoA(\G)}$. For each $n\in\N$ choose strongly continuous, uniformly bounded representations $(\pi_n,\Hil_n)$ of $\G$ and vectors $\xi_n,\eta_n\in\Hil_n$ such that
	\begin{equation*}
		\varphi_n(\y)=\ip{\pi_n(\y)\xi_n}{\eta_n}_{\Hil_n}\qquad(\y\in \G),
	\end{equation*}
	with $\|\pi_n\|\leq\alpha$ and $\|\xi_n\|,\|\eta_n\|\leq1$. Let $\sU$ be an ultrafilter on $\N$ containing all sets $\setw{n\in \N}{n\geq n_0}$ for every $n_0\in \N$, that is, $\sU$ is a free ultrafilter. Let $\Hil$ denote the corresponding ultraproduct of the Hilbert spaces $(\Hil_n)_{n\in \N}$. The elements of $\Hil$ are represented by bounded families $(\zeta_n)_{n\in \N}$, where $\zeta_n\in\Hil_n$, and where two families $(\zeta_n)_{n\in \N}$ and $(\zeta'_n)_{n\in \N}$ defines the same element in $\Hil$ if
	\begin{equation*}
		\lim_{\sU}\|\zeta_n-\zeta'_n\|=0.
	\end{equation*}
	The inner product on $\Hil$ is given by
	\begin{equation*}
		\ip{(\zeta_n)_{n\in \N}}{(\zeta'_n)_{n\in \N}}_\Hil=\lim_\sU\ip{\zeta_n}{\zeta'_n}_{\Hil_n}\qquad((\zeta_n)_{n\in \N},(\zeta'_n)_{n\in \N}\in\Hil).
	\end{equation*}
	Put
	\begin{equation*}
		\xi=(\xi_n)_{n\in \N},\qquad \eta=(\eta_n)_{n\in \N},
	\end{equation*}
	regarded as elements in $\Hil$, and let $(\pi,\Hil)$ be the representation of $\G$ defined by
	\begin{equation*}
		\pi(\y)(\zeta_n)_{n\in \N}=(\pi_n(\y)\zeta_n)_{n\in \N}\qquad(\y\in\G,\,(\zeta_n)_{n\in \N}\in\Hil).
	\end{equation*}
	Then
	\begin{equation*}
		\varphi(\y)=\lim_{n\to\infty}\ip{\pi_n(\y)\xi_n}{\eta_n}_{\Hil_n}=\lim_\sU\ip{\pi_n(\y)\xi_n}{\eta_n}_{\Hil_n}=\ip{\pi(\y)\xi}{\eta}_\Hil
	\end{equation*}
	for $\y\in\G$.
	Furthermore, $\|\pi\|=\sup_{n\in \N}\|\pi_n\|\leq\alpha$ and $\|\xi\|,\|\eta\|\leq1$. Unfortunately, $\pi$ is not necessarily strongly continuous, which we will now remedy.
	
	Let $\Hil'$ be the subspace of $\Hil$ given by
	\begin{equation*}
		\Hil'=\overline{\sspan}{\setw{\pi(g)\xi}{g\in \G}},
	\end{equation*}
	or just $\Hil'=[\pi(\G)\xi]$ for short. Since $\Hil'$ is $\pi(\G)$-invariant, we can define a new representation $(\pi',\Hil')$ of $\G$ by letting
	\begin{equation*}
		\pi'(g)=\pi(g)|_{\Hil'}\qquad(g\in \G).
	\end{equation*}
	Put
	\begin{equation*}
		\xi'=\xi\quad\mbox{and}\quad\eta'=P_{\Hil'}\eta,
	\end{equation*}
	where $P_{\Hil'}$ is the orthogonal projection of $\Hil$ onto $\Hil'$. We find that
	\begin{equation*}
		\varphi(g)=\ip{\pi(g)\xi}{\eta}_{\Hil}=\ip{\pi(g)\xi'}{\eta'}_{\Hil'}=\ip{\pi'(g)\xi'}{\eta'}_{\Hil'}\qquad(g\in \G),
	\end{equation*}
	since $\pi(g)\xi'\in\Hil'$ for every $g\in \G$. Furthermore, $\|\pi'\|\leq\|\pi\|\leq\alpha$, $\|\xi'\|\leq1$, $\|\eta'\|\leq1$ and $[\pi'(\G)\xi']=[\pi(\G)\xi]=\Hil'$.
	
	Let $\Hil''$ be the subspace of $\Hil'$ given by
	\begin{equation*}
		\Hil''=[\pi'(\G)^*\eta'].
	\end{equation*}	
	Since $\Hil''$ is $\pi'(\G)^*$-invariant, we can define a new representation $(\pi'',\Hil'')$ of $\G$ by letting
	\begin{equation*}
		\pi''(g)^*=\pi'(g)^*|_{\Hil''}\qquad(g\in \G).
	\end{equation*}
	Put
	\begin{equation*}
		\xi''=P_{\Hil''}\xi'\quad\mbox{and}\quad\eta''=\eta',
	\end{equation*}
	where $P_{\Hil''}$ is the orthogonal projection of $\Hil'$ onto $\Hil''$. We find that
	\begin{eqnarray*}
		\varphi(g)
		& = & \ip{\pi'(g)\xi'}{\eta'}_{\Hil'}=\ip{\xi'}{\pi'(g)^*\eta'}_{\Hil'}\\
		& = & \ip{\xi''}{\pi'(g)^*\eta''}_{\Hil''}=\ip{\xi''}{\pi''(g)^*\eta''}_{\Hil''}=\ip{\pi''(g)\xi''}{\eta''}_{\Hil''}
	\end{eqnarray*}
	for $g\in \G$, since $\pi'(g)^*\eta''\in\Hil''$. Furthermore, $\|\pi''\|\leq\|\pi'\|\leq\alpha$, $\|\xi''\|\leq1$, $\|\eta''\|\leq1$ and $[\pi''(\G)^*\eta'']=[\pi'(\G)^*\eta']=\Hil''$. Finally,
	\begin{equation*}
		\pi''(g)P_{\Hil''}=P_{\Hil''}\pi'(g)\qquad(g\in \G),
	\end{equation*}
	considered as bounded operators from $\Hil'$ to $\Hil''$, since
	\begin{equation*}
		\pi''(g)P_{\Hil''}=(\pi''(g)^*)^*P_{\Hil''}=(\pi'(g)^*|_{\Hil''})^*P_{\Hil''}\qquad(g\in \G),
	\end{equation*}
	and for arbitrary $\zeta'\in\Hil'$ and $\zeta''\in\Hil''$,
	\begin{eqnarray*}
		\ip{(\pi'(g)^*|_{\Hil''})^*P_{\Hil''}\zeta'}{\zeta''}_{\Hil''}
		& = & \ip{P_{\Hil''}\zeta'}{\pi'(g)^*\zeta''}_{\Hil''}\\
		& = & \ip{\zeta'}{\pi'(g)^*\zeta''}_{\Hil''}\\
		& = & \ip{\pi'(g)\zeta'}{\zeta''}_{\Hil''}\\
		& = & \ip{P_{\Hil''}\pi'(g)\zeta'}{\zeta''}_{\Hil''},
	\end{eqnarray*}
	where we used that $\pi'(g)^*\zeta''\in\Hil''$. It now follows that
	\begin{equation*}
		[\pi''(\G)\xi'']=P_{\Hil''}[\pi'(\G)\xi']=P_{\Hil''}\Hil'=\Hil''.
	\end{equation*}
	
	Since $\varphi\in\MoA(\G)$ is continuous, the function
	\begin{equation*}
		g\mapsto \varphi(g_1 g g_2)=\ip{\pi''(g)\pi''(g_2)\xi''}{\pi''(g_1)^*\eta''}_{\Hil''}\qquad(g\in \G)
	\end{equation*}
	is also continuous for all $g_1,g_2\in \G$. Using this together with $\|\pi''\|\leq\alpha$, $[\pi''(\G)^*\eta'']=\Hil''$ and $[\pi''(\G)\xi'']=\Hil''$ we find that $\pi''$ is \emph{weakly continuous}, i.e., continuous with respect to the weak operator topology on $\linbeg(\Hil'')$. Since $\pi''$ is also uniformly bounded, strong continuity follows automatically by the following argument:
	
	Since $(\pi'',\Hil'')$ is a weakly continuous uniformly bounded representation of $\G$, we can extend it to a representation of the involutive Banach convolution algebra $\ELL^1(\G)$ (this representation will also be called $(\pi'',\Hil'')$, letting the context clarify which one we mean) by setting
	\begin{equation*}
		\pi''(f)=\int_\G f(g)\pi''(g)\dd\mu_\G(g)\qquad(f\in \ELL^1(\G)),
	\end{equation*}
	where the integral converges in the weak operator topology. It is readily checked that
	\begin{equation*}
		\|\pi''(f)\|\leq\alpha\|f\|_1\qquad(f\in \ELL^1(\G)),
	\end{equation*}
	and
	\begin{equation*}
		\pi''(g)\pi''(f)=\pi''(\lambda(g) f)\qquad(g\in \G, f\in \ELL^1(\G)),
	\end{equation*}
	where $\lambda:\G\to\linbeg(\ELL^1(\G))$ is the \emph{left regular representation} given by
	\begin{equation*}
		(\lambda(g)f)(g')=f(g^{-1}g')\qquad(g,g'\in \G,\, f\in\ELL^1(\G)).
	\end{equation*}
	For $f\in \ELL^1(\G)$, $\zeta\in\Hil''$ and $g_0,g\in \G$ we have that
	\begin{equation*}
		\|\pi''(g)\pi''(f)\zeta-\pi''(g_0)\pi''(f)\zeta\|
		\leq\alpha\|\lambda(g)f-\lambda(g_0)f\|_1\|\zeta\|,
	\end{equation*}
	which converges to zero as $g$ converges to $g_0$ by strong continuity of the left regular representation. We put
	\begin{equation*}
		\Hil_0''=\sspan\setw{\pi''(f)\zeta}{f\in \ELL^1(\G),\, \zeta\in\Hil''},
	\end{equation*}
	and conclude that the mapping
	\begin{equation*}
		g\mapsto\pi''(g)\zeta_0\qquad(\zeta_0\in\Hil_0'')
	\end{equation*}
	is continuous on $\G$. We will now show that $\Hil_0''$ is norm dense in $\Hil''$. Let $(f_j)_{j\in J}$ be an approximate unit in $\ELL^1(\G)$ (considered as a Banach convolution algebra), that is, $(f_j)_{j\in J}$ is a net of non-negative norm $1$ functions in $\ELL^1(\G)$, such that for every neighborhood $V$ of $e$, there exists $j_V\in J$ with $\support(f_j)\subseteq V$ for all $j\geq j_V$. Using weak continuity of the representation $(\pi'',\Hil'')$ of $\G$ it is easily seen that $(\pi''(f_j))_{j\in J}$ converges to the identity operator $I\in\linbeg(\Hil'')$ in the weak operator topology. Using that
	\begin{equation*}
		\setw{\pi''(f)}{ f\in \ELL^1(\G) }
	\end{equation*}
	is a convex subset of $\linbeg(\Hil'')$, and therefore has identical closure in the weak- and strong operator topologies, we find a net $(f_j)_{j\in J'}$ in $\ELL^1(\G)$, such that $(\pi''(f_j))_{j\in J'}$ converges to $I$ in the strong operator topology. From this we conclude that $\Hil_0''$ is norm dense in $\Hil''$. Strong continuity of the representation $(\pi'',\Hil'')$ of $\G$ now follows from its uniform boundedness.
\end{proof}
The following realization of the predual of $\MoA(\G)$ is found in~\cite[Proposition~1.10~a)]{DCH:MultipliersOfTheFourierAlgebrasOfSomeSimpleLieGroupsAndTheirDiscreteSubgroups}. If $\Ban_0(\G)$ denotes the completion of $\ELL^1(\G)$ with respect to the ${\Ban_0(\G)}$-norm given by
\begin{equation*}
	\|f\|_{\Ban_0(\G)}=\sup\setw{|\dual{f}{\varphi}|}{\varphi\in\MoA(\G),\,\|\varphi\|_{\MoA(\G)}\leq1}\qquad(f\in\ELL^1(\G)),
\end{equation*}
then the dual space of $\Ban_0(\G)$ is $\MoA(\G)$, and the ${\MoA(\G)}$-norm is the corresponding dual norm. Note that $\|\varphi\|_\infty\leq\|\varphi\|_{\MoA(\G)}$ for $\varphi\in\MoA(\G)$ so it follows that $\|f\|_{\Ban_0(\G)}\leq\|f\|_1$ for $f\in\ELL^1(\G)$.
\begin{proposition}
	\label{cExist}
	If all completely bounded Fourier multipliers of $\G$ are coefficients of strongly continuous, uniformly bounded representations, then the ${\Ban_0(\G)}$-norm is equivalent to the Banach algebra semi-norm $\trip\cdot\trip_\alpha$ for some $\alpha\geq1$. In particular, there exists a constant $c>0$ such that
	\begin{equation*}
		\|f*h\|_{\Ban_0(\G)}\leq c\|f\|_{\Ban_0(\G)}\|h\|_{\Ban_0(\G)}\qquad(f,h\in\ELL^1(\G)).
	\end{equation*}
\end{proposition}
\begin{proof}
	The assumption can be reformulated as
	\begin{equation*}
		\MoA(\G)=\bigcup_{n,m\in\N}m\ST_n.
	\end{equation*}
	But, by Lemma~\ref{ClosedConvex}, $\ST_\alpha$ is a ${\MoA(\G)}$-norm closed subset of $\MoA(\G)$ for $\alpha\geq1$. Hence, by the Baire theorem one of the sets $mS_n$ (and hence $S_n$) for some $n,m\in\N$ must contain an inner point. But according to Lemma~\ref{ClosedConvex} $S_n$ is convex, and since also $S_n=-S_n$ it follows that $0$ is an inner point of $S_n$ and therefore that there exists a $\delta>0$ such that
	\begin{equation*}
		\setw{\varphi\in\MoA(\G)}{ \|\varphi\|_{\MoA(\G)}\leq\delta}\subseteq \ST_n.
	\end{equation*}
	According to Lemma~\ref{ClosedConvex}
	\begin{equation*}
		\ST_n\subseteq\setw{\varphi\in\MoA(\G)}{ \|\varphi\|_{\MoA(\G)}\leq n^2},
	\end{equation*}
	so it follows that
	\begin{equation*}
		\delta\|f\|_{\Ban_0(\G)}\leq\trip f\trip_n\leq n^2\|f\|_{\Ban_0(\G)}\qquad(f\in\ELL^1(\G)),
	\end{equation*}
	and therefore that the two norms are equivalent. The remaining conclusion follows easily with $c=\frac{n^4}{\delta}$, since
	\begin{equation*}
		\trip f*h\trip_n\leq\trip f\trip_n\trip h\trip_n\qquad(f,h\in \ELL^1(\G))
	\end{equation*}
	according to Lemma~\ref{seminorm}.
\end{proof}

We wish to arrive at a way to disprove the existence of such a $c>0$ through knowledge of spherical functions, so from now on we assume that $\G$ is part of a Gelfand pair $(\G,\K)$. Let $\mu_\K$ denote the left and right invariant Haar measure on $\K$, normalized such that $\mu_\K(\K)=1$. For $f\in\ELL^1(\G)$ and $k_1,k_2\in\K$ let ${}_{k_1}f_{k_2}$ denote the translate of $f$ in the sense that
\begin{equation*}
	{}_{k_1}f_{k_2}(g)=f(k_1^{-1}g k_2)\qquad(g\in\G).
\end{equation*}
Note that ${}_{k_1}f_{k_2}\in\ELL^1(\G)$ with $\|{}_{k_1}f_{k_2}\|_1=\|f\|_1$ and that the map
\begin{equation*}
	(k_1,k_2)\mapsto{}_{k_1}f_{k_2}\qquad(k_1,k_2\in\K)
\end{equation*}
is norm continuous. Since $\ELL^1(\G)$ is a Banach space and $\ELL^1(\G)^*$ separate the points (even $\cont_\vanish(\G)$ separate the points) one can use standard vector-valued integration techniques to define
\begin{equation*}
	\ELL^1(\G)\ni f^\natural=\int_{\K\times\K}{}_{k_1}f_{k_2}\dd(\mu_\K\otimes\mu_\K)((k_1,k_2))
\end{equation*}
for $f\in\ELL^1(\G)$, and find that
\begin{equation}
	\label{fradnorm}
	\|f^\natural\|_1\leq\int_{\K\times\K}\|{}_{k_1}f_{k_2}\|_1\dd(\mu_\K\otimes\mu_\K)((k_1,k_2))=\|f\|_1.
\end{equation}
We will refer to $f^\natural$ as the \emph{radialization} of $f$. Similarly, one wishes to define a radialization $\varphi^\natural$ of $\varphi\in\MoA(\G)$ (cf.~\cite[Proposition~1.6~(a)]{CH:CompletelyBoundedMultipliersOfTheFourierAlgebraOfASimpleLieGroupOfRealRankOne}). To this end, we need to know that ${}_{k_1}\varphi_{k_2}\in\MoA(\G)$ with $\|{}_{k_1}\varphi_{k_2}\|_{\MoA(\G)}=\|\varphi\|_{\MoA(\G)}$ for $k_1,k_2\in\K$, where ${}_{k_1}\varphi_{k_2}$ is the translate of $\varphi$. But this follows easily using Proposition~\ref{Gilbert0}. Note that
\begin{equation*}
	|\dual{f}{{}_{k_1}\varphi_{k_2}}-\dual{f}{{}_{k_1'}\varphi_{k_2'}}|\leq\|{}_{k_1^{-1}}f_{k_2^{-1}}-{}_{k_1'^{-1}}f_{k_2'^{-1}}\|_1\|\varphi\|_{\MoA(\G)}
\end{equation*}
for $f\in\ELL^1(\G)$ and $k_1,k_2,k_1',k_2'\in\K$. Since $\Ban_0(\G)^*=\MoA(\G)$ and $\ELL^1(\G)$ is a dense subset of $\Ban_0(\G)$ one now finds that the map
\begin{equation*}
	(k_1,k_2)\mapsto{}_{k_1}\varphi_{k_2}\qquad(k_1,k_2\in\K)
\end{equation*}
is $w^*$ continuous. Since $\|{}_{k_1}\varphi_{k_2}\|_{\MoA(\G)}=\|\varphi\|_{\MoA(\G)}$ for $k_1,k_2\in\K$ it follows that $\setw{{}_{k_1}\varphi_{k_2}}{k_1,k_2\in\K}$ is a norm bounded subset of $\MoA(\G)$, and therefore that $\overline{\co}^{w^*}\setw{{}_{k_1}\varphi_{k_2}}{k_1,k_2\in\K}$ (the $w^*$ closed convex hull) is a $w^*$ closed, norm bounded subset of $\MoA(G)$. From Alaoglu's theorem it finally follows that $\overline{\co}^{w^*}\setw{{}_{k_1}\varphi_{k_2}}{k_1,k_2\in\K}$ is $w^*$ compact. Using this together with the fact that $\MoA(\G)$, equipped with the $w^*$ topology, is a topological vector space whose dual separates the points (the dual with respect to the $w^*$ topology is $\Ban_0(\G)$) one can use~\cite[Theorem~3.27]{Rud:FunctionalAnalysis} to define
\begin{equation*}
	\MoA(\G)\ni\varphi^\natural=\int_{\K\times\K}{}_{k_1}\varphi_{k_2}\dd(\mu_\K\otimes\mu_\K)((k_1,k_2))
\end{equation*}
for $\varphi\in\MoA(\G)$, from which it follows by standard arguments that
\begin{equation}
	\label{varphiradnorm}
	\|\varphi^\natural\|_{\MoA(\G)}\leq\int_{\K\times\K}\|{}_{k_1}\varphi_{k_2}\|_{\MoA(\G)}\dd(\mu_\K\otimes\mu_\K)((k_1,k_2))=\|\varphi\|_{\MoA(\G)}.\footnote{Actually, it follows from~\cite[Theorem~3.27]{Rud:FunctionalAnalysis} that $\varphi^\natural\in\overline{\co}^{w^*}\setw{{}_{k_1}\varphi_{k_2}}{k_1,k_2\in\K}$ from which the desired result also follows, since we have already shown that this set is bounded in norm by $\|\varphi\|_{\MoA(\G)}$.}
\end{equation}
One can easily show that $\dual{f^\natural}{\varphi}=\dual{f}{\varphi^\natural}$ for $f\in\ELL^1(\G)$ and $\varphi\in\ELL^\infty(\G)$ (recall that $\G$ is unimodular). Since obviously $(\varphi^\natural)^\natural=\varphi^\natural$ one gets the following relations between the two types of radialization:
\begin{equation}
	\label{connection}
	\dual{f^\natural}{\varphi}=\dual{f^\natural}{\varphi^\natural}=\dual{f}{\varphi^\natural}.
\end{equation}
\begin{lemma}
	\label{anotherrad}
	If $f\in\ELL^1(\G)$, then $\|f^\natural\|_{\Ban_0(\G)}\leq\|f\|_{\Ban_0(\G)}$.
\end{lemma}
\begin{proof}
	The lemma follows from~\eqref{connection} and~\eqref{varphiradnorm} since $\MoA(\G)$ is the dual of $\Ban_0(\G)$.
\end{proof}
\begin{proposition}
	\label{!cExist}
	If there is a constant $c>0$ such that
	\begin{equation}
		\label{constc}
		\|f*h\|_{\Ban_0(\G)}\leq c\|f\|_{\Ban_0(\G)}\|h\|_{\Ban_0(\G)}\qquad(f,h\in\cont_\cpt(\G)^\natural),
	\end{equation}
	then any spherical function $\varphi$ in $\MoA(\G)$ will satisfy $\|\varphi\|_{\MoA(\G)}\leq c$.
\end{proposition}
\begin{proof}
	Assume the existence of a $c>0$ satisfying~\eqref{constc}. Consider the algebra $\cont_\cpt(\G)^\natural$ of finitely supported radial functions on $\G$ with multiplication given by convolution. This is a commutative algebra and it follows from the assumption that the completion of $\cont_\cpt(\G)^\natural$ under the ${\Ban_0(\G)}$-norm is a commutative Banach algebra with respect to the norm $c\|\cdot\|_{\Ban_0(\G)}$. Let $\varphi$ be a spherical function which is also a completely bounded Fourier multiplier of $\G$. On the one hand, since $\varphi$ is a spherical function we have that
	\begin{equation}
		\label{character}
		f\mapsto\dual{f}{\varphi}\qquad(f\in\cont_\cpt(\G)^\natural)
	\end{equation}
	is a character (cf.~\cite[Lemma~1.5]{FTP:HarmonicAnalysisOnFreeGroups}). On the other hand, since $\varphi\in\MoA(\G)$, we have by duality that
	\begin{equation*}
		|\dual{f}{\varphi}|\leq\|\varphi\|_{\MoA(G)}\|f\|_{\Ban_0(\G)}\qquad(f\in\cont_\cpt(\G)^\natural),
	\end{equation*}
	so we can extend~\eqref{character} to a character on $\overline{\cont_\cpt(\G)^\natural}^{\|\cdot\|_{\Ban_0(\G)}}$. But since every character on an (abelian) Banach algebra has norm less than or equal to $1$, we have that
	\begin{equation*}
		|\dual{f}{\varphi}|\leq c\|f\|_{\Ban_0(\G)}\qquad(f\in\cont_\cpt(\G)^\natural).
	\end{equation*}
	Notice that $\cont_\cpt(\G)$ is dense in $\ELL^1(\G)$ with respect to the ${\ELL^1(\G)}$-norm and therefore with respect to the ${\Ban_0(\G)}$-norm and therefore also dense in $\Ban_0(\G)$. Using this together with duality, Lemma~\ref{anotherrad} and~\eqref{connection} (recall that $\varphi^\natural=\varphi$) we find that
	\begin{eqnarray*}
		\|\varphi\|_{\MoA(\G)}
		& = & \sup\setw{ |\dual{f}{\varphi}| }{ f\in\cont_\cpt(\G),\, \|f\|_{\Ban_0(\G)}\leq1}\\
		& = & \sup\setw{ |\dual{f^\natural}{\varphi}| }{ f\in \cont_\cpt(\G),\, \|f\|_{\Ban_0(\G)}\leq1}\\
		& \leq & \sup\setw{ |\dual{h}{\varphi}| }{ h\in \cont_\cpt(\G)^\natural,\, \|h\|_{\Ban_0(\G)}\leq1}\\
		& \leq & c.
	\end{eqnarray*}
\end{proof}
\begin{maintheorem}
	\label{HaageruponGelfand}
	Let $\G$ be a group of the form $\GG$, $SU(1,n)$, $Sp(1,n)$ (with $n\geq2$), $F_{4(-20)}$ or $\PGLQ$ (with $q$ a prime number). There is a completely bounded Fourier multiplier of $\G$ which is not the coefficient of a strongly continuous, uniformly bounded representation of $\G$.
\end{maintheorem}
\begin{proof}
	When $\G$ is $\GG$, $SU(1,n)$, $Sp(1,n)$ (with $n\geq2$) or $F_{4(-20)}$ the theorem follows from Theorem~\ref{uniformlyunbd} and Proposition~\ref{cExist} and~\ref{!cExist}. The case $\G=\PGLQ$ follows by replacing Theorem~\ref{uniformlyunbd} with~\cite[Theorem~5.8]{HSS:SchurMultipliersAndSphericalFunctionsOnHomogeneousTrees}, which states that there is no uniform bound on the ${\MoA(\G)}$-norm among the spherical functions on $\PGLQ$ which are completely bounded Fourier multipliers of $\PGLQ$ (the norms are explicitly calculated).
\end{proof}
\begin{remark}
	\label{withoutsoctsnow}
	Actually, using the techniques of the proof of Lemma~\ref{ClosedConvex}, one can verify that strong continuity can be omitted in Theorem~\ref{HaageruponGelfand}.
\end{remark}

\bigskip\noindent
Let $\Gamma$ be a group of the form
\begin{equation}
	\label{freeconvolution1}
	\Gamma=\FF,
\end{equation}
where $M,N\in\N_0$ with $M+2N\geq3$. In particular, this includes the groups
\begin{equation*}
	\ast_{m=1}^M\Z/2\Z\qquad(3\leq M<\infty)
\end{equation*}
and the (non-abelian) free groups
\begin{equation*}
	\F_N=\ast_{n=1}^N\Z\qquad(2\leq N<\infty).
\end{equation*}
Let $e$ denote the identity element in $\Gamma$ and put $q=M+2N-1$. By~\cite[p.~16--18]{FTN:HarmonicAnalysisAndRepresentationTheoryForGroupsActingOnHhomogeneousTrees} the Cayley graph of $\Gamma$ is a homogeneous tree of degree $q+1$. We now work toward obtaining the following result, stating that the conclusion of Theorem~\ref{HaageruponGelfand} also holds for $\Gamma$. The proof follows the methods from an unpublished manuscript of U.~Haagerup for the case $\Gamma=\F_N$. A different proof for the case $\Gamma=\F_N$ was later found by Pisier (cf.~\cite{Pis:AreUnitarizableGroupsAmenable?}).
\begin{theorem}
	\label{Haagerup}
	Consider a group $\Gamma$ of the form~\eqref{freeconvolution1}. There is a completely bounded Fourier multiplier of $\Gamma$ which is not the coefficient of a uniformly bounded representation of $\Gamma$.
\end{theorem}
The proof of Theorem~\ref{HaageruponGelfand} (in the case when $\G$ is $\GG$, $SU(1,n)$, $Sp(1,n)$ (with $n\geq2$) or $F_{4(-20)}$) followed from Theorem~\ref{uniformlyunbd} and Proposition~\ref{cExist} and~\ref{!cExist}. We will show that these three results are still true, when one replaces $\G$ with $\Gamma$. Obtaining these three results for $\Gamma$ was the approach taken by U.~Haagerup in his unpublished manuscript. We will now go through the argumentation needed to verify these three results for $\Gamma$. Proposition~\ref{cExist} was proved for locally compact groups, so this still holds true for $\Gamma$ (the proof of Lemma~\ref{ClosedConvex} for $\Gamma$ is in fact considerably easier, since the part about strong continuity can be omitted). The analogue of Theorem~\ref{uniformlyunbd} for $\Gamma$ follows from~\cite[Theorem~4.4]{HSS:SchurMultipliersAndSphericalFunctionsOnHomogeneousTrees} in which the actual ${\MoA(\Gamma)}$-norm of the spherical functions on $\Gamma$ are calculated (the spherical functions on $\Gamma$ are not given in terms of Gelfand pairs, but this will be taken up shortly).
What remains in order to prove Theorem~\ref{Haagerup}, is to prove Proposition~\ref{!cExist} for $\Gamma$. To do this, we recall the definition of the spherical functions on $\Gamma$.

Let $\dd:\Gamma\times\Gamma\to\N_0$ be the graph distance on the Cayley graph of $\Gamma$ (note that $\dd$ is invariant under left multiplication). A function $f:\Gamma\to\C$ is called \emph{radial} if there exists a function $\dot f:\N_0\to\C$ such that
\begin{equation*}
	f(x)=\dot f(\dd(x,e))\qquad(x\in\Gamma).
\end{equation*}
Since $\dd(x,e)$ is the reduced word length of a $x\in\Gamma$, we will often write $|x|$ instead of $\dd(x,e)$. Let $\cont_\cpt(\Gamma)^\natural$ denote the finitely supported functions on $\Gamma$ which are radial (we let the superscripts $\natural$ on a set of functions on $\Gamma$ denote the subset consisting of the radial functions). It is well known that $\cont_\cpt(\Gamma)^\natural$ is commutative with respect to convolution (cf.~\cite[Ch.~3~Lemma~1.1]{FTP:HarmonicAnalysisOnFreeGroups}). Analogously to the case of Gelfand pairs, a function $\varphi\in\cont(\Gamma)^\natural$ is called a \emph{spherical function} on $\Gamma$ if
\begin{equation*}
	f\mapsto\dual{f}{\varphi}=\sum_{x\in\Gamma}f(x)\varphi(x)\qquad(f\in\cont_\cpt(\Gamma)^\natural)
\end{equation*}
is a non-zero character (cf.~\cite[Ch.~3~Lemma~1.5]{FTP:HarmonicAnalysisOnFreeGroups}). The bounded spherical functions are exactly those which extend to characters of $\ell^1(\Gamma)^\natural$.
By going through the proof of Proposition~\ref{!cExist} (and Lemma~\ref{anotherrad}) it is seen that everything works out in the case of $\Gamma$ if we can establish a radialization $f\mapsto f^\natural$ of functions on $\Gamma$ satisfying the following formulas corresponding to~\eqref{fradnorm}, \eqref{varphiradnorm} and~\eqref{connection}:
\begin{equation}
	\label{3.6'}
	\|f^\natural\|_1\leq\|f\|_1\qquad(f\in\ell^1(\Gamma)),
\end{equation}
\begin{equation}
	\label{3.7'}
	\|\varphi^\natural\|_{\MoA(\Gamma)}\leq\|\varphi\|_{\MoA(\Gamma)}\qquad(\varphi\in\MoA(\Gamma))
\end{equation}
and
\begin{equation}
	\label{3.8'}
	\dual{f^\natural}{\varphi}=\dual{f^\natural}{\varphi^\natural}=\dual{f}{\varphi^\natural}\qquad(f\in\ell^1(\Gamma),\,\varphi\in\ell^\infty(\Gamma)).
\end{equation}
Put
\begin{equation}
	\label{EEn}
	\EE_n=\setw{x\in\Gamma }{ |x|=n }\qquad(n\in\N_0)
\end{equation}
and
\begin{equation}
	\label{a13.11}
	\EE_x=\EE_{|x|}=\setw{y\in\Gamma}{|y|=|x|}\qquad(x\in\Gamma).
\end{equation}
For $h:\Gamma\to\C$ define $h^\natural:\Gamma\to\C$ by
\begin{equation}
	\label{3.175}
	h^\natural(x)=\frac{1}{|\EE_x|}\sum_{y\in \EE_x}h(y)\qquad(x\in\Gamma),
\end{equation}
where $|\EE|$ denotes the number of elements in a set $\EE$. It is obvious that $h^\natural$ is radial and the reader may verify~\eqref{3.8'} using the same technique as for verifying~\eqref{connection}. Let $f\in\ell^1(\Gamma)$ and note that
\begin{equation*}
	\|f^\natural\|_1=\sum_{n\in\N_0}|\sum_{y\in\EE_n}f(y)|\leq\sum_{n\in\N_0}\sum_{y\in\EE_n}|f(y)|=\|f\|_1,
\end{equation*}
which verifies~\eqref{3.6'}. Establishing~\eqref{3.7'} requires more effort and is postponed until Proposition~\ref{3.3.10}.

Let $\K$ be the group of isometries of the Cayley graph of $\Gamma$ leaving invariant the identity element $e\in\Gamma$. Then $\K$ is a subgroup of the infinite product
\begin{equation*}
	\prod_{n=0}^\infty\SK(\EE_n)
\end{equation*}
of the permutation groups $\SK(\EE_n)$ of $\EE_n$.
Each $\SK(\EE_n)$ is finite and hence compact in the discrete topology. By the Tychonoff theorem, $\prod_{n=0}^\infty\SK(\EE_n)$ is compact in the product topology. Since the product topology on $\prod_{n=0}^\infty\SK(\EE_n)$ coincide with the topology of pointwise convergence, one easily gets that $\K$ is a closed (and hence compact) subgroup. Let $\mu_\K$ denote the normalized left and right invariant Haar measure on $\K$.
\begin{remark}
	We mention that $\K$ is part of a Gelfand pair $(\G,\K)$, for which $\Gamma$ is isomorphic (as a set) to $\G/\K$ and the spherical functions on $\Gamma$ are in one-to-one correspondence with the spherical functions on $(\G,\K)$. In fact, $\G$ is given by the isometries of the Cayley graph of $\Gamma$, cf.~\cite[Ch.~3~\S~V]{FTP:HarmonicAnalysisOnFreeGroups} and \cite{Dun:EtudeD'uneClasseDeMarchesAl`eatoiresSurL'arbreHomog'ene}.
\end{remark}
Note that $\K$ acts transitively on the sets $\EE_n$ for $n\in\N_0$. Hence, for $x,y\in\EE_n$ the measure under $\mu_\K$ of
\begin{equation*}
	\setw{k\in\K}{k(x)=y}
\end{equation*}
is independent of $y$ (for fixed $x$) and therefore
\begin{equation*}
	\mu_\K(\setw{k\in\K}{k(x)=y})=\frac{1}{|\EE_n|}\qquad(x,y\in\EE_n).
\end{equation*}
From this, it follows that
\begin{equation*}
	h^\natural(x)=\int_\K h(k(x))\dd\mu_\K(k)\qquad(x\in\Gamma).
\end{equation*}
\begin{lemma}
	\label{LemmaA}
	If $s\in\EE_x$ and $t\in\EE_y$ satisfy $t^{-1}s\in\EE_{y^{-1}x}$ for $x,y\in\Gamma$, then there exists $k\in\K$ such that $k(x)=s$ and $k(y)=t$.
\end{lemma}
\begin{proof}
	Put $m=|x|$, $n=|y|$ and $l=\tfrac{1}{2}(m+n-\dd(x,y))$. Then $l\in\N_0$ and the reduced word of $y^{-1}x$ is obtained by canceling the last $l$ letters in $y^{-1}$ and the first $l$ letters in $x$. Therefore, 
	\begin{equation*}
		x=u x'\quad\mbox{and}\quad y=u y',
	\end{equation*}
	where $|u|=l$, $|x'|=m-l$ and $|y'|=n-l$, and where $y'^{-1}x'$ is reduced. Similarly,
	\begin{equation*}
		s=v s'\quad\mbox{and}\quad t=v t',
	\end{equation*}
	where $|v|=l$, $|s'|=m-l$ and $|t'|=n-l$, and where $t'^{-1}s'$ is reduced. See Figure~\ref{rotation} for an illustration of the shortest routes between points $(e,x,y)$ and $(e,s,t)$.
	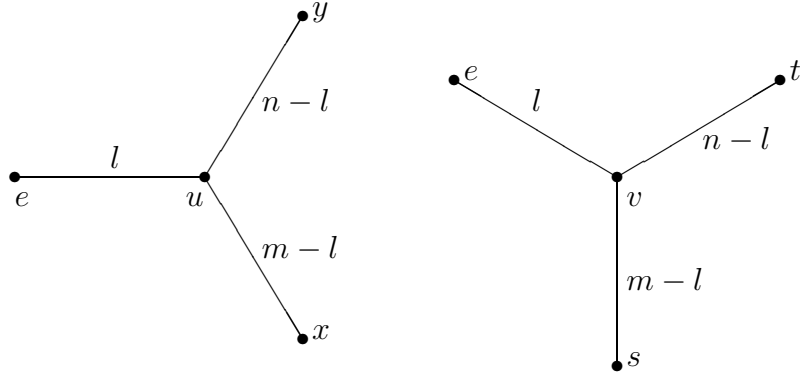
\begin{figure}
		\begin{minipage}[t b c]{\textwidth}
			\begin{center}
				\begin{tabular}{cc}
					\begin{picture}(2,2)
						\put(1,1){\line(3,5){0.5145}}
						\put(1,1){\line(3,-5){0.5145}}
						\put(1,1){\line(-1,0){1}}
						\put(1,1){\circle*{0.05}}
						\put(0,1){\circle*{0.05}}
						\put(1.5145,1.8575){\circle*{0.05}}
						\put(1.5145,0.1425){\circle*{0.05}}
						\put(0,0.85){$e$}
						\put(0.9,0.85){$u$}
						\put(0.5,1.05){$l$}
						\put(1.5645,0.1425){$x$}
						\put(1.3,0.5713){$m-l$}
						\put(1.5645,1.8575){$y$}
						\put(1.3,1.35){$n-l$}
					\end{picture}
					&
					\begin{picture}(2,2)
						\put(1,1){\line(5,3){0.8575}}
						\put(1,1){\line(-5,3){0.8575}}
						\put(1,1){\line(0,-1){1}}
						\put(1,1){\circle*{0.05}}
						\put(1,0){\circle*{0.05}}
						\put(1.8575,1.5145){\circle*{0.05}}
						\put(0.1425,1.5145){\circle*{0.05}}
						\put(1.05,0){$s$}
						\put(1.05,0.85){$v$}
						\put(1.05,0.4){$m-l$}
						\put(0.1925,1.5145){$e$}
						\put(0.55,1.35){$l$}
						\put(1.9075,1.5145){$t$}
						\put(1.45,1.15){$n-l$}
					\end{picture}
				\end{tabular}
			\end{center}
			\caption{Shortest route between the points $(e,x,y)$ and $(e,s,t)$.}
			\label{rotation}
		\end{minipage}
	\end{figure}
	Particularly, the three shortest routes from $u$ to $e$, $x$ and $y$ starts on three different edges from $u$, and the same holds for the three routes from $v$ to $e$, $s$ and $t$. Since the Cayley graph of $\Gamma$ is a homogeneous tree, there exists $k\in\K$ such that
	\begin{equation*}
		k(u)=v,\quad k(x)=s\quad\mbox{and}\quad k(y)=t,
	\end{equation*}
	which proves the lemma. In the above argument it was implicitly assumed that $l\geq1$, $m-l\geq1$ and $n-l\geq1$, but it is easy to see that a modified argument can be used if this is not the case.
\end{proof}
\begin{lemma}
	\label{lemmaB}
	For every function $h:\Gamma\to\C$
	\begin{equation*}
		h^\natural(y^{-1}x)=\int_\K h(k(y)^{-1}k(x))\dd\mu_\K(k)\qquad(x,y\in\Gamma).
	\end{equation*}
\end{lemma}
\begin{proof}
	Consider two fixed elements $x,y\in\Gamma$. Notice that
	\begin{equation*}
		\dd(k(x),k(y))=\dd(x,y)\qquad(k\in\K)
	\end{equation*}
	and therefore
	\begin{equation*}
		|k(y)^{-1}k(x)|=|y^{-1}x|\qquad(k\in\K),
	\end{equation*}
	which by~\eqref{a13.11} can be expressed as
	\begin{equation*}
		k(y)^{-1}k(x)\in\EE_{y^{-1}x}\qquad(k\in\K).
	\end{equation*}
	Put
	\begin{equation*}
		\A_z=\setw{k\in\K}{k(y)^{-1}k(x)=z}\qquad(z\in\EE_{y^{-1}x}).
	\end{equation*}
	Then $\K$ is equal to the disjoint union
	\begin{equation}
		\label{a13.13}
		\K=\bigsqcup_{z\in\EE_{y^{-1}x}}\A_z.
	\end{equation}
	Hence,
	\begin{equation}
		\label{a13.14}
		\int_\K h(k(y)^{-1}k(x))\dd\mu_\K(k)=\sum_{z\in\EE_{y^{-1}x}}h(z)\mu_\K(\A_z).
	\end{equation}
	Thus, in order to prove Lemma~\ref{lemmaB} we must show that
	\begin{equation*}
		\mu_\K(\A_z)=\frac{1}{|\EE_{y^{-1}x}|}\qquad(z\in\EE_{y^{-1}x}).
	\end{equation*}
	Put now
	\begin{equation*}
		\BB=\setw{(s,t)\in\Gamma\times\Gamma}{|s|=|x|,\,|t|=|y|,\,|t^{-1}s|=|y^{-1}x|}.
	\end{equation*}
	Then $\BB$ is a finite subset of $\Gamma\times\Gamma$, which is invariant under the action of $\K$ on $\Gamma\times\Gamma$ given by
	\begin{equation*}
		(s,t)\mapsto(k(s),k(t))\qquad(k\in\K,\, s,t\in\Gamma).
	\end{equation*}
	Moreover, by Lemma~\ref{LemmaA} this action is transitive on $\BB$. Therefore, each of the sets
	\begin{equation*}
		\A_{s,t}=\setw{k\in\K}{k(x)=s,\,k(y)=t}\qquad((s,t)\in\BB)
	\end{equation*}
	has the same Haar measure in $\K$, and since $\K$ is the disjoint union of all the sets $A_{s,t}$,
	\begin{equation}
		\label{a13.15}
		\mu_\K(\A_{s,t})=\frac{1}{|\BB|}\qquad((s,t)\in\BB).
	\end{equation}
	Put
	\begin{equation*}
		\BB_z=\setw{(s,t)\in\BB}{t^{-1}s=z}\qquad(z\in\EE_{y^{-1}x}).
	\end{equation*}
	Then
	\begin{equation*}
		\A_z=\bigsqcup_{(s,t)\in\BB_z}\A_{s,t}\qquad(z\in\EE_{y^{-1}x}),
	\end{equation*}
	so by~\eqref{a13.15}
	\begin{equation*}
		\mu_\K(\A_z)=\frac{|\BB_z|}{|B|}\qquad(z\in\EE_{y^{-1}x}).
	\end{equation*}
	Let $\unit_{\EE_x}$ and $\unit_{\EE_y}$ denote the characteristic functions of ${\EE_x}$ and ${\EE_y}$, respectively. Then for $z\in\EE_{y^{-1}x}$,
	\begin{eqnarray*}
		|\BB_z| & = & |\setw{(s,t)\in\EE_x\times\EE_y}{t^{-1}s=z}|\\
		& = & |\setw{t\in\Gamma}{t\in\EE_y,\,t z\in\EE_x}|\\
		& = & \sum_{t\in\Gamma}\unit_{\EE_y}(t)\unit_{\EE_x}(t z) \\
		& = & (\unit_{\EE_y}*\unit_{\EE_x})(z),
	\end{eqnarray*}
	where we have used that $\EE_y^{-1}=\EE_y$. Since the set $\cont_\cpt(\Gamma)^\natural$ of radial functions with finite support on $\Gamma$ form an abelian algebra with respect to convolution, $\unit_{\EE_y}*\unit_{\EE_x}$ is radial, and hence $|\BB_z|$ is independent of $z\in\EE_{y^{-1}x}$. Thus, by~\eqref{a13.13} we have
	\begin{equation*}
		\mu_\K(\A_z)=\frac{1}{|\EE_{y^{-1}x}|}\qquad(z\in\EE_{y^{-1}x}),
	\end{equation*}
	which together with~\eqref{a13.14} proofs Lemma~\ref{lemmaB}.
\end{proof}
\begin{proposition}
	\label{3.3.10}
	If $\varphi\in\MoA(\Gamma)$, then $\varphi^\natural\in\MoA(\Gamma)$ and
	\begin{equation*}
		\|\varphi^\natural\|_{\MoA(\Gamma)}\leq\|\varphi\|_{\MoA(\Gamma)}.
	\end{equation*}
\end{proposition}
\begin{proof}
	Assume that $\varphi\in\MoA(\Gamma)$ and use Proposition~\ref{Gilbert0} to find a Hilbert space $\Hil$ and bounded maps $\PP,\QQ:\Gamma\to\Hil$ such that
	\begin{equation*}
		\varphi(y^{-1}x)=\ip{\PP(x)}{\QQ(y)}_{\Hil}\qquad(x,y\in\Gamma)
	\end{equation*}
	and
	\begin{equation*}
		\|\PP\|_\infty\|\QQ\|_\infty=\|\varphi\|_{\MoA(\Gamma)}.
	\end{equation*}
	Put $\widetilde\Hil=\ELL^2(K,\Hil,\mu_\K)$ and define $\widetilde\PP,\widetilde\QQ:\Gamma\to\widetilde\Hil$ by
	\begin{equation*}
		(\widetilde\PP(x))(k)=\PP(k(x))\qquad(x\in\Gamma,\,k\in\K)
	\end{equation*}
	and
	\begin{equation*}
		(\widetilde\QQ(y))(k)=\QQ(k(y))\qquad(y\in\Gamma,\,k\in\K).
	\end{equation*}
	For fixed $x,y\in\Gamma$ the maps $\widetilde\PP(x)$ and $\widetilde\QQ(y)$ are continuous and therefore measurable. Moreover, the norms of $\widetilde\PP(x)$ and $\widetilde\QQ(y)$ satisfy
	\begin{equation*}
		\|\widetilde\PP(x)\|^2=\int_{\K}\|\PP(k(x))\|_{\Hil}^2\dd\mu_\K(k)\leq\|\PP\|_\infty^2\qquad(x\in\Gamma)
	\end{equation*}
	and
	\begin{equation*}
		\|\widetilde\QQ(y)\|^2=\int_{\K}\|\QQ(k(y))\|_{\Hil}^2\dd\mu_\K(k)\leq\|\QQ\|_\infty^2\qquad(y\in\Gamma).
	\end{equation*}
	According to Lemma~\ref{lemmaB},
	\begin{equation*}
			\varphi^\natural(y^{-1}x)=\int_\K\ip{\PP(k(x))}{\QQ(k(y))}_{\Hil}\dd\mu_\K(k)=\ip{\widetilde\PP(x)}{\widetilde\QQ(y)}_{\widetilde\Hil}\qquad(x,y\in\Gamma),
	\end{equation*}
	from which we conclude, using Proposition~\ref{Gilbert0}, that $\varphi^\natural\in\MoA(\Gamma)$ with
	\begin{equation*}
		\|\varphi^\natural\|_{\MoA(\Gamma)}\leq\|\widetilde\PP\|_\infty\|\widetilde\QQ\|_\infty\leq\|\PP\|_\infty\|\QQ\|_\infty=\|\varphi\|_{\MoA(\Gamma)}.
	\end{equation*}
\end{proof}
We have now established~\eqref{3.6'}--\eqref{3.8'} and therefore finished the proof of Theorem~\ref{Haagerup}.
\begin{corollary}
	\label{inftycase}
	Consider a countable discrete group $\Gamma'$ which has a subgroup $\Gamma$ of the form~\eqref{freeconvolution1}. There is a completely bounded Fourier multiplier of $\Gamma'$ which is not the coefficient of a uniformly bounded representation of $\Gamma'$.
\end{corollary}
\begin{proof}
	Let $\varphi$ be a completely bounded Fourier multiplier of $\Gamma$ which is not the coefficients of any uniformly bounded representation of $\Gamma$. Let $\varphi'$ be the extension of $\varphi$ to $\Gamma'$ by zero outside $\Gamma$. According to Bo{\.z}ejko and Fendler (cf.~\cite[Lemma~1.2]{BF:Herz-SchurMultipliersAndUniformlyBoundedRepresentationsOfDiscreteGroups}) $\varphi'$ is a completely bounded Fourier multiplier of $\Gamma'$. If $\varphi'$ was the coefficient of a uniformly bounded representation $(\pi',\Hil')$ of $\Gamma'$, then the restriction of this representation to $\Gamma$ would give a contradiction with the choice of $\varphi$.
\end{proof}
\begin{remark}
	From Corollary~\ref{inftycase} it follows in particular that there is a completely bounded Fourier multiplier of $\F_\infty$ which is not the coefficient of a uniformly bounded representation of $\F_\infty$, where $\F_\infty$ is the free group on infinitely many generators.
\end{remark}

	\subsection*{Acknowledgments}
The author wishes to express his thanks to Uffe Haagerup for many helpful conversations and for granting access to his unpublished manuscript.

	\bibliography{troelsBibliography}
	\contrib{Troels Steenstrup}{troelssj@imada.sdu.dk}\smallskip\\
Department of Mathematics and Computer Science, University of Southern Denmark, Campusvej~55, DK--5230~Odense~M, Denmark.\\

\end{document}